\documentclass[a4paper,11pt]{article}
\usepackage{amssymb}
\usepackage{cite}
\usepackage{color}
\usepackage{amsmath}
\usepackage{enumitem}
\usepackage{multirow}
\usepackage{booktabs}  
\usepackage{array}     
\usepackage{microtype}

\newcommand{\M}{{\cal M}}

\newcommand{\cL}{{\cal L}}

\newcommand{\R}{\mathbb{R}}
\newcommand{\Z}{\mathbb{Z}}
\newcommand{\N}{\mathbb{N}}

\newcommand{\cO}{{\cal O}}

\newcommand{\cuad}{{\sqcap\kern-.68em\sqcup}}

\newcommand{\bH}{{\mathbb{H} }}
\newcommand{\cE}{{\cal E}}
\newcommand{\cJ}{{\cal J}}
\newcommand{\cH}{{\cal H}}
\newcommand{\cW}{{\cal W}}

\newtheorem{theorem}{Theorem}[section]
\newtheorem{proposition}{Proposition}[section]

\newtheorem{lemma}{Lemma}[section]
\newtheorem{corollary}{Corollary}[section]
\newtheorem{remark}{Remark}[section]
\newcommand{\bremark}{\begin{remark} \em}
	\newcommand{\eremark}{\end{remark} }

\begin{document}
	
\begin{center}{\bf \large     Poisson problem for the Hardy operator  on lattice graphs\\[1.5mm]
and application to semilinear equations

		}\bigskip 
		\bigskip
		
		{\small
			Huyuan Chen\footnote{chenhuyuan@simis.cn, chenhuyuan@yeah.net} \qquad Bobo Hua\footnote{bobohua@fudan.edu.cn}  \qquad
			Wendi Xu\footnote{xuwendi@simis.cn} 
			\bigskip\medskip
		 
		\medskip

		$ ^{1}$   Center for Mathematics and Interdisciplinary Sciences, Fudan University, \\ 
		Shanghai 200433, China\\[2pt]
		Shanghai Institute for Mathematics and Interdisciplinary Sciences,\\ 
		Shanghai 200433, China\\[14pt] 
		$ ^2$ School of Mathematical Sciences, LMNS, Fudan University,\\
		Shanghai, 200433, P.R. China\\[2pt]
		Shanghai Center for Mathematical Sciences, Fudan University,\\ Shanghai 200438, China
		\\[14pt]
		$ ^{3}$ Shanghai Institute for Mathematics and Interdisciplinary Sciences,\\ 
		Shanghai 200433, China\\[2pt]
		Research Institute of Intelligent Complex Systems, Fudan University, \\
		Shanghai 200433, China\\[18pt]
		}
		\begin{abstract}

We investigate the analytic properties of the Hardy-type operator $\mathcal{L}_{\mu}:=-\Delta + \mu H_0$ on the integer lattice graph $\mathbb{Z}^d$, $d\geq3,$ where $\mu \geq -1$ and the nonnegative potential $H_0$ approaches the optimal lattice Hardy weight from below.
We derive sharp fixed-pole Green kernel estimates on $\mathbb{Z}^d$ and obtain a necessary and sufficient weighted summability condition characterizing the solvability of the nonnegative Poisson problem associated with this operator. 

Based on this classification, we establish Liouville nonexistence ranges for
$\cL_\mu u\ge W u^p$, $p>0$, under a lower power bound on $W$.
Under an upper power bound, we construct positive supersolutions with specified power or logarithmic profiles; additional two-sided control yields
positive solutions of the equation in the stated sublinear ranges.
\end{abstract}
		
\end{center}
\smallskip
	
\smallskip

\vspace{2mm}

\setcounter{equation}{0}
\section{Introduction}
		
Let the integer lattice graph $\Z^d$, $d\geq 3$, be equipped with the counting measure and unit edge weight.  Every vertex has degree $2d$ and $x\sim y$ if and only if $|x-y|_{_Q}:=  \sum^d_{k=1}|x_k-y_k|=1$.
Throughout this paper, we consider real and finite-valued functions, and use the unnormalized Laplacian
$$\Delta u(x)=\sum_{y\sim x}\big(u(y)-u(x)\big),\quad\  \forall x\in\mathbb{Z}^d.$$

  Let  $\Phi_{d}$ be the minimal positive Green function of $-\Delta$ with pole at origin, 
 i.e.
		$$ -\Delta \Phi_{d} =\delta_0\quad
			{\rm in}\ \  \Z^d, \qquad
			\phantom{   }
			\displaystyle \lim_{|x|\to+\infty}\Phi_{d}(x)=0. 
		$$
The optimal Hardy inequality  from \cite[Theorem 7.1]{KPP} (also see \cite{KPP1}) has 
the form that  
\begin{equation}\label{Htineq}
\frac{1}{2}\sum_{x\in\Z^d}  \sum_{y\sim x} \big( \varphi(x)-\varphi(y)\big)^2 \geq \sum_{z\in\Z^d}  \overline{H}_0(z)\varphi(z)^2, \quad \varphi \in C_c(\Z^d),
\end{equation}
where 
\begin{equation}\label{Green 1}
\overline{H}_0(x) =-\big[\Delta\Phi_{d}^{\frac12}(x)\big]\Phi_{d}^{-\frac12}(x)  >0.
\end{equation}
Furthermore,  by the asymptotics of lattice Green kernel, we have
\begin{equation}\label{ome-1}
\overline{H}_0(x)=\frac{(d-2)^2}{4}|x|^{-2}+O(|x|^{-3})\quad {\rm as}\ \, |x|\to+\infty,
\end{equation}
see \cite{KPP,U}.
A survey on related Hardy and Rellich inequalities   can be found in \cite{Ff,HY24,KL16,KPP1}
and the reference therein. \smallskip

The Hardy inequality implies a critical linear operator.
Our question is how its characteristic spatial scales determine Poisson solvability and then enter a nonlinear inequality to produce explicit Liouville exponents. This leads us to the Hardy-type operator 
\begin{equation*} 
	\cL_\mu =-\Delta +\mu H_0 \quad {\rm on}\ \, \Z^d,  
\end{equation*}
where  $\mu \in \R$ and
$H_0\in C(\Z^d)$  satisfies  one of the following properties
\begin{itemize}
\item[ $(\cH_0)$]      \begin{equation*}
0\leq  H_0(x)\leq \overline{H}_0(x)\ \ {\rm in}\ \, \Z^d\qquad {\rm and}\qquad \lim_{x\in\Z^d,\, |x|\to+\infty} \frac{ H_0(x) }{\overline{H}_0(x)}=1;  
\end{equation*}
			
\item[ $(\cH_1)$]     
\begin{equation*} 
	0\leq  H_0(x)\leq \overline{H}_0(x)\ \ {\rm in}\ \,\Z^d
\end{equation*}
and  for some $\epsilon_0\in(0,1)$,  
\begin{equation*} 
	\limsup_{x\in\Z^d,\, |x|\to+\infty} \frac{\big(\overline{H}_0(x)-H_0(x)\big)(1+|x|)^{\epsilon_0}}{\overline{H}_0(x)} \in [0,+\infty).  
\end{equation*}
\end{itemize}

Note that $(\cH_1)$ is equivalent to that there exists a constant $C>0$ such that  
\begin{equation*} 
    \overline{H}_0\big(1 - C\rho(x)^{-\epsilon_0}\big)_+ \leq H_0 \leq \overline{H}_0 \qquad \text{ in }  \Z^d,
\end{equation*}
and $(\cH_1)$ is strictly stronger than $(\cH_0)$. Under this strengthened assumption, we shall obtain sharp decay estimates for the fundamental solutions.
		 		
For $\mu \ge -1$, define
\begin{equation*}
\tau_\pm=\tau_\pm(\mu):=\frac{2-d}{2}(1\mp \sqrt{1+\mu}).   
\end{equation*}
Here $\tau_+ +\tau_-=2-d$, and the two roots coincide at $\mu=-1$, denoted by
$\tau_c=(2-d)/2$. 		
We use the regularized functions 
$$\rho(x):= (1+|x|^2)^{\frac{1}{2}}, \qquad \ell(x) := \log(e+|x|^2).$$
Thus, $\rho(x) \asymp 1+|x|$ and $\ell \ge 1$. 
For nonnegative functions $f$ and $g $, the notation $ f \asymp g$ in $D$ means that there exist constants $0<c<C$, such that $cf(x) \le g(x) \le Cf(x)$ for all $x \in D$.

Our first main theorem concerns the fundamental solution  of 
\begin{equation}\label{eq 5.1-fw1}
\left\{\arraycolsep=1pt
  \begin{array}{lll}
\cL_\mu  u =\delta_{\xi} \quad
&{\rm in}\ \  \Z^d, \\[2mm]
\phantom{-\   }
\displaystyle u\geq 0\quad
&{\rm in}\ \  \Z^d,
\end{array}
\right.
\end{equation}
for  some fixed $\xi\in\Z^d$.  	
\begin{theorem}\label{pr fun-fw}
	Let $d\geq 3$ and $\xi\in \Z^d$. Assume that one of the following conditions holds:
    \begin{enumerate}[label=(\alph*)]
        \item $\quad \mu > -1$, and   $H_0 $ verifies  $(\cH_0)$;
        \item $\quad \mu \geq -1$, and  $H_0 $ verifies  $(\cH_1)$ with $$H_0(x) < \overline{H}_0(x) \qquad \text{for every } \ x\in\Z^d$$ 
        when $\mu=-1$.
    \end{enumerate}
	Then problem (\ref{eq 5.1-fw1}) has a minimal positive solution $\Phi_{d,\mu}(\cdot,\xi)$ satisfying
    $$\Phi_{d,\mu}(x,\xi) \to 0 \qquad
    \text{ as }  
    |x|\to\infty.$$
    The Green kernel is symmetric:
$$\Phi_{d,\mu}(x, y) = \Phi_{d,\mu}(y, x),\qquad
\forall x,y \in \Z^d.$$

Under assumption $(a)$, for every $0<\theta<\tau_+ - \tau_-$, there are constants $C_{\theta,\xi} \geq 1$ and $R_{\theta,\xi} \geq 1$ such that
	\begin{equation}\label{2.1-E-bf1}
	C_{\theta,\xi}^{-1} \rho(x-\xi)^{\tau_- -\theta} \leq \Phi_{d,\mu}(x,\xi)\leq  C_{\theta,\xi} \,\rho(x-\xi)^{\tau_-+\theta}, \qquad |x-\xi| \geq R_{\theta,\xi}.
	\end{equation}
Under assumption $(b)$, for each fixed pole, there exists a $C_{\xi}\geq 1$ which may depend on $\xi,d,\mu,H_0$, such that
	\begin{equation}\label{2.1-E-bf}
		C_{\xi}^{-1}\rho(x-\xi)^{\tau_-} \leq \Phi_{d,\mu}(x,\xi)\leq  C_{\xi}\,\rho(x-\xi)^{\tau_-},\qquad \forall x\in\Z^d.
	\end{equation}
The solution is unique among positive solutions satisfying the upper bound in \eqref{2.1-E-bf1} or \eqref{2.1-E-bf}.

Moreover,
if $\mu\geq0$, then there is also a uniform upper bound:
\begin{equation}\label{2.1-F-bf}
	0<\Phi_{d,\mu}(x,\xi)\leq \Phi_{d}(x-\xi)\leq C_d \rho(x-\xi)^{2-d}, \quad \forall x\in\Z^d,
	\end{equation}
where $C_d>0$ depends only on $d$.
\end{theorem}

	Minimality is understood among all finite-valued nonnegative
	solutions of the same point-source problem
	\eqref{eq 5.1-fw1}. It does not imply uniqueness among all such
	solutions, even under the restriction that $\displaystyle\lim_{|x|\to+\infty}u(x)=0$ for $\mu\in(-1,0)$.    
	
	Under $(\mathcal H_0)$  and
	$\mu>-1$, only the estimates with an arbitrarily small loss
	in the exponent are asserted.
	In the precise sense, \eqref{2.1-E-bf1} gives the asymptotic behavior
	\begin{equation*}
		\lim_{|x-\xi| \to \infty} \frac{\log \Phi_{d,\mu}(x,\xi)}{\log|x-\xi|} = \tau_-.
	\end{equation*}
	Equivalently, 
	\begin{equation}\label{H-0-decay}
		\Phi_{d,\mu}(x,\xi) = |x-\xi|^{\tau_- + o(1)}.
	\end{equation}

	For $\mu>-1$, assumption $(\mathcal H_0)$ includes the case
	$H_0=\overline H_0$, while for $\mu =-1$, we require $H_{0}<\overline{H}_0$ pointwise in $\Z^d$. In fact, the exactly critical case
	$\mu=-1$ and $H_0=\overline H_0$ is fundamentally different:
	the operator $
	\overline L :=-\Delta-\overline H_0
	$
	is critical on the full lattice and has ground state
	$h=\sqrt{\Phi_d}$, with $\overline Lh=0$ on the entire lattice.
	In particular, it admits no positive Green kernel,
	see \cite{KPP,KPPcrit}.

In $\R^d$, the Hardy operator takes the form  
$$
\cL_{\R^d,\mu} := -\Delta_{\R^d} + \mu\frac{(d-2)^2}{4|x|^2},\qquad \mu\geq -1.
$$ 
The inverse-square potential has the same scaling order as the Laplacian, exhibiting singularity at the origin and criticality at infinity. These properties underpin the central role of the Hardy-type operator in both mathematical analysis and mathematical physics. Hardy–Leray-type problems governed by such operators arise naturally in diverse contexts: molecular physics \cite{LL}, quantum cosmological models—including the Wheeler–DeWitt equation \cite{BE}—and combustion theory \cite{G}.  

In contrast, on the lattice, the origin is a regular vertex and the inverse-square behavior of the Hardy weight concerns infinity.
Consequently, asymptotic behavior at infinity determines the decay profile of fundamental solutions. For $\mu > -1$, the fundamental solution decays like $|x|^{\tau_-(\mu)}$ both on $\R^d$ and $\Z^d$. However, in the extremal case $\mu = -1$, the decay differs qualitatively; on $\R^d$, it exhibits a logarithmic correction, taking the form $|x|^{\tau_c}\ln|x|$, whereas on $\Z^d$, it decays purely as $|x|^{\tau_-}$. 

Furthermore, for the radial homogeneous equation
$\mathcal L_{\mathbb R^d,\mu}u=0$
on $\mathbb R^d\setminus\{0\}$, the characteristic powers are
$r^{\tau_-},r^{\tau_+}$ when $\mu>-1$; at $\mu=-1$ the independent profiles are
$r^{\tau_c}$ and $r^{\tau_c}\log r$.  
Whether an analogous weakly singular solution like $|x|^{\tau_+(\mu)}$ exists for $\cL_\mu$ on the lattice remains open for $\mu\in[-1,+\infty)\setminus\{0\}$. 
\medskip

		Our second main result characterizes the solvability of Poisson problems involving the Hardy operator
		\begin{equation}\label{eq 3.1-p}
			\left\{\arraycolsep=1pt
			\begin{array}{lll}
				\cL_\mu  u    =  g   \quad
				&{\rm in}\ \  \Z^d , \\[2mm]
				\phantom{  -\ }
				\displaystyle  u    \geq 0 &{\rm in}\ \  \Z^d, 
			\end{array}
			\right.
		\end{equation}
		where  $g: \Z^d\to[0,+\infty)$. 
		For $\mu \geq -1$, write 
        $$\mathcal{M}_{\mu} (g) = \sum_{x\in \Z^d} g(x)(1+|x|)^{\tau_-}.$$
\begin{theorem}\label{teo 1-p}
	Let $d\geq 3$.
    Assume that  $H_0 $ verifies $(\cH_1)$
	and  $g\in C(\Z^d)$ is nonnegative.
\begin{enumerate}[label=(\roman*)]
 \item Suppose $\mu\ge-1$ and assume more that 
       \begin{equation}\label{con ext} 
				H_0<\overline{H}_0\quad{\rm in}\ \Z^d\ \,{\rm if}\ \mu=-1.
	    \end{equation}
The following are equivalent:
        \begin{enumerate}[label=(\alph*), leftmargin=*]
            \item $\mathcal{M}_{\mu}(g)<\infty;$
            \item there exists $u\ge0$ with $\cL_\mu u\ge g$; 
            \item Poisson problem
		\eqref{eq 3.1-p} has a finite-valued minimal nonnegative solution.
        \end{enumerate}
In this case the minimal solution is
			\[
			v_g(x)=\sum_{y\in\Z^d}\Phi_{d,\mu}(x,y)g(y).
			\]
			It is strictly positive if $g\not\equiv0$, and equals zero if $g\equiv0$.
\item Suppose $\mu\ge-1$ and $\mathcal M_\mu(g)=\infty$. Even the inequality $\cL_\mu u\ge g$ has no nonnegative supersolution.
\item Suppose $\mu<-1$ and $g\not\equiv0$. Even the inequality $\cL_\mu u\ge g$ has no nonnegative supersolution.
\end{enumerate}
\end{theorem}
		
Note that Theorem $\ref{teo 1-p}(i)$ completely characterizes nonnegative
solvability of the Poison problem \eqref{eq 3.1-p} under the stated assumptions. 
Whenever condition $(a)$ holds, the minimal solution is given by the Green potential and is strictly positive if $g\not\equiv0$.
The weight exponent $\tau_-(\mu)$ is determined by the decay of the Green kernel, which in turn reflects the asymptotic behavior of the Hardy potential at infinity.
This gives a direct link between the asymptotic Hardy scale and the admissibility of nonnegative source terms.
Together with the accompanying Poisson estimates, the solvability criterion and its associated divergence obstruction provide the linear tools used below to derive Liouville-type nonexistence results and construct positive supersolutions of semilinear inequalities.
		
Recently, semilinear elliptic problems on graphs attracts  more and more attention
see  \cite{CH,KLS,HLW,HLY,GL,LW,LLY} and the references therein.  Authors in \cite{BMP} studied the Liouville theorem for the linear equation with a potential
$$\Delta u-Vu=0\quad {\rm in}\ \ G, $$
where $G$ is an infinite weighted graph and $V$ is a super critical potential.  
Related Liouville theorem for the semilinear problem
$$\Delta u+Vu^p=0\quad {\rm in}\ \ G $$
is studied in \cite{MPS,GHS,GHJ,GS} 
under suitable assumptions on potential $V$ and $p>1$ being subcritical.  
		
Motivated by the classification in Theorem \ref{teo 1-p}, 
		we  finally  show the existence and nonexistence of  solution of semilinear   elliptic  inequality 
		\begin{equation}\label{semilinear pro}
			\left\{\arraycolsep=1pt
			\begin{array}{lll}
				\cL_\mu  u\geq W u^p \quad
				&{\rm in}\ \  \Z^d, \\[2mm]
				\phantom{ -\  }
				  u\gneqq 0  \quad  &{\rm   in} \ \  \Z^d, 
			\end{array}
			\right.
		\end{equation}
		where $p>0$,  $\mu\geq-1$ and $W,\, H_0\in C(\Z^d)$ are positive  potentials.  We  find different types of Serrin's critical exponents according to 
		the Hardy potential and coefficient $\mu$,  build the nonexistence in the subcritical cases and construct a supersolution in the supercritical cases. 
		
		\smallskip

        In this paper, we first construct a fundamental solution to the associated linear operator. Based on this construction, we establish a sharp necessary and sufficient condition on the nonhomogeneous term that guarantees the existence of a minimal positive solution to the Poisson equation. We then apply this criterion to investigate the Liouville property—that is, the nonexistence of nontrivial nonnegative supersolutions—for a class of weighted semilinear inequalities on the entire lattice space.  The rest of this paper is organized as follows. In Section 2, we analyze the basic properties of Hardy operator, including  the comparison principle. 
		In Section 3, we prove the existence of a fundamental solution to the Hardy operator in $\Z^d$. 
		In Section 4, we give  the classification of Poisson problems.    Section 5 is devoted to establishing the existence and nonexistence results for positive supersolutions to the semilinear problem \eqref{semilinear pro}.
		
		\setcounter{equation}{0}
		\section{ Preliminary  }
		
\subsection{Notation and finite comparison principle}
            
For $x=(x_{1}, \cdots ,x_{d})\in\Z^d$, denote the Euclidean norm and the graph-distance norm, respectively, by
$$|x| := (\sum_{i=1}^{d}x_{i}^2)^{\frac{1}{2}}\qquad \text{ and } \qquad |x|_Q=\sum_{i=1}^d|x_i|.$$
        
Given a subset  $\Omega \subset \Z^d$, the exterior boundary is defined as 
$$\partial \Omega  = \{y\notin \Omega : \exists x\in \Omega , \text{ s.t. } x \sim y \},$$ 
and the interior boundary as 
$$\delta \Omega  = \{x \in \Omega : \exists y\notin \Omega , \text{ s.t. } x \sim y \}.$$ 
Write the closure of $\Omega $ as $\overline{\Omega }:= \Omega \cup \partial \Omega $. 
		We use $C(\Omega )$ for the set of functions with real value in $\Omega $ and $C_0(\Omega )$ for the set of
		functions in $\overline{\Omega }$ with zero values in $\partial \Omega $.
		
		The associated energy form of Laplacian is, for all $f,g\in C(\Z^d)$, 
		\[\mathcal{E}(f,g)
		:=\frac{1}{2}\sum_{x\in\Z^d}\sum_{\substack{y\in \Z^d  \\x\sim y}}
		(f(x)-f(y) )(g(x)-g(y)),
		\]
		and we write $\mathcal{E}(f)= \mathcal{E}(f,f)$.
		By restricting the energy form to a subset $\Omega $, we have
		\[
		\mathcal{E}_\Omega (f,g)
		:=
		\frac{1}{2}\sum_{x\in \Omega}\sum_{\substack{y\in \Omega \\x\sim y}}
		(f(x)-f(y) )(g(x)-g(y))
		+
		\sum_{x\in \Omega }\sum_{\substack{y\in\partial \Omega \\x\sim y}}
		(f(x)-f(y) )(g(x)-g(y)).
		\]
		
\begin{lemma}\label{Green formula}
Let $D\subset \Z^d$ be a finite subset. For any $f\in C(\Z^d)$ and $g\in C_{0}(D)$,  we have discrete Green formula
$$ \cE_D(f,g)
= \sum_{x\in D}(-\Delta f)(x)g(x).
$$
\end{lemma}
		
\noindent{\bf Proof. } 
Symmetrization of the interior sum gives
\begin{align*}
\mathcal{E}_D(f,g)
&=\frac12\sum_{x\in D}
	\sum_{\substack{y\in D\\y\sim x}}
	(f(x)-f(y))(g(x)-g(y))
+
 \sum_{x\in D}\sum_{\substack{y\in\partial D\\y\sim x}}(f(x)-f(y))(g(x)-g(y))\\
&=
\sum_{x\in D}
\sum_{\substack{y\in D\\y\sim x}}
(f(x)-f(y))g(x)
+
\sum_{x\in D}\sum_{\substack{y\in\partial D\\y\sim x}}(f(x)-f(y))(g(x)-g(y))\\
&=
\sum_{x\in D}(-\Delta f)(x)g(x)
-
\sum_{x\in D}\sum_{\substack{y\in\partial D\\y\sim x}}(f(x)-f(y))g(y),
\end{align*}
and the second term vanishes when $g=0$ on $\partial D$.
\hfill$\Box$\medskip
		
Recall that   $-\Delta h=\overline{H}_0 h$ with  $h=\sqrt{\Phi_d}$. Direct expansion of squares gives the ground-state identity
\begin{equation}\label{eq:ground-state}
\mathcal{E}(\varphi)-\sum_x\overline{H}_0(x)\varphi(x)^2
=\frac12\sum_x\sum_{y\sim x}h(x)h(y)
\left(\frac{\varphi(x)}{h(x)}-\frac{\varphi(y)}{h(y)}\right)^2,
\quad\varphi\in C_c(\Z^d).
\end{equation}
The right-hand side is nonnegative. If it equals to 0, then $\varphi/h$ is a constant on $\Z^d$ due to the connectivity, and finite support would make that constant zero.
Therefore, the right-hand side is strictly positive for nonzero $\varphi \in C_c(Z^d)$.
	
The following is the comparison and strong maximum principle in finite domain. 
\begin{lemma}\label{com lm}
Assume $\mu\ge-1$ and $0\le H_0\le\overline{H}_0$.
Let $\Omega \in \Z^d$ be a finite subset and $w\in C(\overline{\Omega})$ satisfying
\begin{equation}\label{eq 2.1 cm}
\cL_\mu  w  \geq  0  \quad {\rm in}\ \Omega ,
\qquad
w\geq 0 \quad {\rm on}  \,  \partial  \Omega, 
\end{equation}
then $w\geq 0$ in $\Omega$.
If $\Omega$ is connected, then
either $w\equiv 0$ in $\Omega$ or  $u>0$ in $\Omega$. 
\end{lemma}
\noindent{\bf Proof. }  
Extend $w_{-}:=\max\{0,-w\}$ by zero outside $\Omega$. Direct computation shows that 
\[(a-b)(a_{-} - b_{-})  \leq -(a_{-}-b_{-})^2.\]
Multiply (\ref{eq 2.1 cm}) by $ w_-$ and sum over $\Omega$. Combining Green formula Lemma \ref{Green formula} with the equality above, we obtain the following.  
\begin{equation}\label{comp-123}
\begin{aligned}
0&\leq \sum_{x\in\Omega}\cL_{\mu}w(x)w_{-}(x)\\
&=\mathcal{E}_{\Omega}(w,w_-)
-\mu\sum_{x\in\Omega}H_0(x)w_-^2(x)\\
&\leq-\Big(\mathcal{E}(w_-)
+\mu\sum_{x\in\Z^d}H_0(x)w_-^2(x)\Big)\\
&=-\Big(\mathcal{E}(w_-) -\sum_{x\in\Z^d}\overline{H}_0(x)w_-^2(x)\Big)- \sum_{x}\big( \overline{H}_0+\mu H_0
 \big)(x)(w_-)^2(x).
\end{aligned}
\end{equation}
By Hardy inequality (\ref{Htineq}),  \,
$\mathcal{E}(w_-) -\sum_{x\in\Omega}\overline{H}_0(x)w_-^2(x)\geq0,
$\,
and the assumptions give\\	
$
\overline{H}_0(x)+\mu H_0(x)=\overline{H}_0(x)-H_0(x)+(\mu+1)H_0(x)\geq 0.
$
We conclude that each inequality in \eqref{comp-123} is equality. 
In particular,
we have 
\[\mathcal{E}(w_-) -\sum_{x\in\Z^d}\overline{H}_0(x)w_-^2(x)=0.
\]
By \eqref{eq:ground-state}, this is strictly positive unless $w_-\equiv0$, which proves $w\geq0$. 

Assume that $x_0\in\Omega$ is such that $w(x_0)=0$, then
\[
0\leq \cL_{\mu}w(x_0)=
\sum_{y\sim x_0}
-w(y) \leq 0.
\]
As a consequence,
$w(y)=0$ for all $y\sim x_0.$
Repeating the argument along the connected graph, we obtain $w\equiv0$ in $\Omega$.
\hfill$\Box$\medskip 

The following lemma transforms the existence of a positive supersolution into the nonnegativity of quadratic form, serving as a direct tool for proving non-existence.
\begin{lemma}\label{picone-ineq}
Let $\Omega$ be a subset in $\Z^d$ and $V$ be a real potential. 
If
\begin{equation*}
   (-\Delta+V)u\geq0 \quad \text{ in } \Omega,\qquad
u>0 \, \text{ on } \overline{\Omega}, 
\end{equation*}
then
\begin{equation}\label{eq:picone}
		\mathcal{E}(\varphi)+\sum_x V(x)\varphi(x)^2
        \ge 0,
         \qquad \forall \varphi\in C_c(\Omega).
\end{equation}
More generally, if $(-\Delta+V)u \ge f \ge0$, then the right-hand side of \eqref{eq:picone} can be replaced by
$\sum_x\frac{f(x)}{u(x)}\varphi(x)^2.$
\end{lemma}

\noindent{\bf Proof. } 
We prove the general case with a nonnegative source term $f$; the first assertion follows by taking $f=0$.

For $\varphi\in C_c(\Omega)$, multiply $(-\Delta+V)u \ge f$ by $\varphi^2/u$, and the Green formula gives
\begin{equation*}
    \begin{aligned}
     &\sum_{x} f(x)\frac{\varphi^2}{u}(x)\\
    &\leq \sum_{x} (-\Delta u  +V u )(x) \frac{\varphi(x)^2}{u(x)} \\
     &= \mathcal{E} (u,\frac{\varphi^2}{u})+ \sum_{x} V(x) \varphi(x)^2 \\
     &=\mathcal{E} (\varphi) + \sum_{x} V(x) \varphi(x)^2
     -\frac{1}{2}\sum_{x}\sum_{y\sim x}u(x)u(y)\left(\frac{\varphi(x)}{u(x)}-\frac{\varphi(y)}{u(y)}\right)^2,
    \end{aligned}
\end{equation*}
which ends proof. \hfill$\Box$\medskip

\subsection{Power and logarithmic expansions}
Next  we   provide some basic calculations for power and logarithmic test functions, which is very useful in the derivation of the fundamental solutions and in the construction of super and subsolutions for semilinear models. 

 For $\sigma\in\R$, define 
$$\varphi_\sigma (t)=(1+t)^\frac{\sigma}{2}\ \, {\rm for}\ t\geq0,\qquad \psi_{\sigma}(x):= \varphi_{\sigma}(|x|^2) = \rho(x)^{\sigma}.$$
We regard $\psi_\sigma $ as a smooth function on $\R^d$ and then restrict it to $\Z^d$. Let $e_1,\ldots,e_d$ denote the standard coordinate vectors. For each $i$, Taylor expansion in the two opposite directions gives
$$\psi_\sigma(x+e_i)+\psi_\sigma(x-e_i)-2\psi_\sigma(x)
=
\partial_{ii}\psi_\sigma(x)
+O\bigl(\rho(x)^{\sigma-4}\bigr),$$
where the remainder follows from
$
\sup_{|t|\le1}
\left|\partial_i^4\psi_\sigma(x+te_i)\right|
\le C_{\sigma,d}\rho(x)^{\sigma-4}.
$
Summing over $i$, we obtain
$$
\Delta\psi_\sigma(x)
=
\Delta_{\mathbb R^d}\psi_\sigma(x)
+O\bigl(\rho(x)^{\sigma-4}\bigr).
$$
Since
\begin{equation*}
\begin{aligned}
	\Delta_{\mathbb R^d}\psi_\sigma(x)
	&=
	d\sigma\rho(x)^{\sigma-2}
	+\sigma(\sigma-2)|x|^2\rho(x)^{\sigma-4}\\
	&=
	\sigma(\sigma+d-2)\rho(x)^{\sigma-2}
	-\sigma(\sigma-2)\rho(x)^{\sigma-4},
\end{aligned}
\end{equation*}
we obtain
$$
\Delta\psi_\sigma(x)
=
\sigma(\sigma+d-2)\rho(x)^{\sigma-2}
+O\bigl(\rho(x)^{\sigma-4}\bigr)
\qquad\text{as }|x|\to\infty.
$$

Under assumption $(\mathcal{H}_0)$, the asymptotic behavior of $H_0$ therefore yields
\begin{equation}\label{est-1}
 \cL_\mu\psi_\sigma(x)
= -\Delta \psi_\sigma(x) +\mu H_0(x) \psi_\sigma(x) =
\beta_0(\sigma)\rho(x)^{\sigma-2}
+o\bigl(\rho(x)^{\sigma-2}\bigr),
\end{equation}
where
$$
\beta_0(\sigma):=
-\sigma(\sigma+d-2)+\mu\frac{(d-2)^2}{4}
= -(\sigma-\tau_-)(\sigma-\tau_+).
$$
Under the stronger assumption $(\mathcal{H}_1)$, we have
$ H_0(x)=
\frac{(d-2)^2}{4}\rho(x)^{-2}
+O\bigl(\rho(x)^{-2-\epsilon_0}\bigr).
$
Since $ 0<\epsilon_0<1$, it follows that
\begin{equation}  \label{est-1-im}
\mathcal L_\mu\psi_\sigma(x)
=
\beta_0(\sigma)\rho(x)^{\sigma-2}
+O\bigl(\rho(x)^{\sigma-2-\epsilon_0}\bigr).
\end{equation}

When $\sigma=\tau_\pm(\mu)$, the leading-order coefficient $\beta_0(\sigma)$ in $\cL_\mu \psi_\sigma $ vanishes.
We therefore introduce logarithmically corrected test functions  to control the sign of the resulting expression. Recall that we write $\ell(x)=\log(e+|x|^2).$

 For $\sigma, s\in \R$, define 
$$F(t):= (1+t)^{\frac{\sigma}{2}} \big(\log(e+t)\big)^s, \ t\ge 0, \qquad
\chi(x) := F(|x|^2)=\rho(x)^{\sigma} \ell(x)^{s}, \ x\in \Z^d.$$	
Then we have
\begin{align*}
		\Delta_{\Z^d} \chi(x) 
		&= \sum_{y\sim x} \chi(y) - \chi(x)\\
		&=\sum_{y\sim x} \big[(|y|^{2}-|x|^2))F'(|x|^2) + \frac{1}{2}(|y|^2-|x|^2)^2 F''(|x|^2)\big](1+o(1))\\
		&=2d\left\{\frac{\sigma}{2} \rho^{\sigma-2}\ell^s + s\rho^{\sigma}\ell^{s-1}\frac{1}{|x|^2+e} \right\} \\
      & \,+(4|x|^2+d) \left\{ \frac{\sigma}{2}(\frac{\sigma}{2}-1)\rho^{\sigma-4}\ell^s +\frac{\sigma s\rho^{\sigma-2}\ell^{s-1}}{|x|^2+e} +\frac{s\rho^{\sigma}\left[(s-1)\ell^{s-2}-\ell^{s-1}\right]}{(|x|^{2}+e)^2}\right\}(1+o(1))\\
        &=\sigma(\sigma-2+d)|x|^{\sigma-2}\ell^{s} + 2s(2\sigma-2+d)|x|^{\sigma-2}\ell^{s-1} + 4s(s-1)|x|^{\sigma-2}\ell^{s-2} +O(|x|^{\sigma-4}\ell^{s}).
	\end{align*}
	Under assumption $(\cH_1)$, we obtain, as $|x|\to\infty$,
		\begin{align}\label{eq:log-expansion}
			\cL_\mu\chi(x)
			= &|x|^{\sigma-2}\bigl[\beta_0(\sigma)\ell^s
			-2s(2\sigma+d-2)\ell^{s-1}-4s(s-1)\ell^{s-2}\bigr] 
			 +O\bigl(|x|^{\sigma-2-\epsilon_0}\ell^s\bigr).
		\end{align}
		
\setcounter{equation}{0}
\section{Fundamental solutions for the Hardy operator}
		
We construct the fundamental solution by exhaustion with finite balls and then establish its fixed-pole estimates. 
Fix a center $\bar{\xi}\in\Z^d$, independent of the pole $\xi$, and set $\Omega_n = B_{n}(\bar{\xi}):=\{x\in\Z^d:|x-\bar{\xi}|\le n\}$. For $n$  large enough with $\xi\in\Omega_n$,
consider:
\begin{equation}\label{bdd fundamental}
 \begin{cases}
  \cL_\mu u=\delta_\xi \quad 
  & \text{in }\Omega_n,\\
  u=0 
  & \text{in } \Z^d\setminus\Omega_n.
 \end{cases}
\end{equation}
We remark that the Dirichlet boundary condition is equivalent to 
$$u=0 \quad\ \  \text{on } \partial \Omega_n. $$
		
\begin{lemma}\label{lm fun-bf-1}
Let the assumptions of  Theorem  \ref{pr fun-fw} hold. For every $n\in\N$ such that $\xi\in\Omega_n$, problem (\ref{bdd fundamental}) has a unique  solution $v_{n}(\cdot,\xi)$, which is positive in $\Omega_n$. The mapping $n\mapsto v_{n}(\cdot,\xi)$ is increasing.
Futhermore,
\begin{align}  \label{2.1-b-h1}
v_{n}(\xi,\xi) \leq \frac1{ \overline{H}_0(\xi) +\mu    H_0(\xi) }  
\end{align} 
and, for each fixed $x\in\Z^d$,
\begin{align}  \label{2.1-b-h1-1}
0\le v_n(x,\xi)\leq C_{x,\xi}  ,
\end{align}
where the constant $C_{x,\xi}>0$ may depends on $x, \xi,d,\mu$ and $ H_{0}$,
but is independent of $n$.	Moreover, when $\mu\geq0$, 
\begin{align}  \label{2.1-b-h2}
0\leq v_{n}(x,\xi)\leq v_{n}(\xi,\xi),   \qquad \forall x \in \Z^d. 
\end{align} 	
\end{lemma}

\noindent
{\bf Proof. }  
Note that, under the assumptions on $\mu$ and $H_{0}$,
\begin{equation}\label{positv}
(\overline{H}_0+\mu H_0)(x) > 0,\quad \forall x\in  \Z^d.
\end{equation}
The uniqueness follows by the comparison principle.

\smallskip
\noindent {\textit{ Existence and monotonicity.} } 
Denote the quadratic energy form associated with the Hardy operator by
$$\mathcal{E}_{\mu}(v) = \mathcal{E}(v) + \mu\sum_{x\in \Z^d}  H_{0}(x) v^2(x). $$
For a finite set $\Omega \subset \Z^d$,  define 
$$\bH_{0,\mu}(\Omega)=\Big\{v\in C(\Z^d):\ v=0\ \ {\rm in}\ \,\Z^d\setminus \Omega,\  \mathcal{E}_{\mu}(v) <+\infty \Big\}. $$
This is a finite-dimensional vector space, and the embedding $\bH_{0,\mu}(\Omega)\subset L^{2}(\Omega)$ is compact. 
Define
$$\cJ_{\mu}(v) :=\frac{1}{2}\mathcal{E}_{\mu}(v)-v(\xi),  \quad \forall  v\in \bH_{0,\mu}(\Omega_n),$$
as the energy functional related to (\ref{bdd fundamental}). The Hardy inequality \eqref{Htineq} together with \eqref{positv} show that
\begin{align*}
\cJ_{\mu}(v) 
& \geq \frac{1}{2}\sum_{x\in \Z^d}  \big( \overline{H}_0(x)+ \mu H_{0}(x) \big) v^2(x)-v(\xi) \\
& \geq \sum_{x\in \Omega} C_{\Omega}v^2(x) - v(\xi)\\
&\to \infty \quad \text{as} \quad \|v\|_{2} \to \infty.
\end{align*}
Therefore, $\cJ_{\mu}$ is continuous, coercive, and strictly convex  on $\bH_{0,\mu}(\Omega_n)$. Hence, it admits a unique minimizer, denoted by $v_n(\cdot,\xi)$.
The Euler-Lagrange equation and Lemma \ref{Green formula} give
\begin{equation}\label{weakform}
\sum_{x\in \Omega}
\left( -\Delta v_n(x) + \mu H_0(x)v_n(x)
\right) \phi(x)
= \phi(\xi)
\end{equation}
for every $\phi \in \bH_{0,\mu}(\Omega_n)$.
Taking the test functions supported at individual vertices yields
\begin{equation} \label{bddeq}
-\Delta v_n (\cdot, \xi)+\mu H_0v_n (\cdot, \xi)=\delta_\xi
\quad \text{in }  \Omega_n.
\end{equation}
Furthermore, from the comparison principle  Lemma \ref{com lm},  we have that 
$v_n>0$ in $\Omega_n$ and   the mapping $n\mapsto v_{n}(\cdot,\xi)$ is increasing.   \smallskip

\noindent {\textit{ Uniform pointwise bounds:}}
Let $\phi = v_{n}$ in (\ref{weakform}), we obtain
\begin{equation*}   
v_n (\xi,\xi) = \mathcal{E}_{\mu}(v_{n}) 
\geq \sum_{x\in \Omega}  \big( \overline{H}_0(x)+ \mu H_{0}(x) \big) v_{n}^2(x,\xi).
\end{equation*}
The term at $x=\xi$ and the positivity of $v_{n}(\xi,\xi)$
implies \eqref{2.1-b-h1}.
For each $z\in \Omega_n $, there exists a shortest path in $\Omega$ connecting $\xi$ and $z$, denoted by
$$ x_0 \sim x_1 \sim \cdots\sim x_{l},$$
where $x_{0}=\xi$ and $ x_{l}= z$. 
Equation \eqref{bddeq}) is equivalent to 
$$\delta_{\xi}(x)
=\cL_\mu v_{n}(x,\xi) 
=(2d+\mu H_{0}(x))v_{n}(x,\xi) - \sum_{y\sim x} v_{n}(y,\xi)$$
Set $a(x):=2d+\mu H_0(x)$. By the definition of $\overline H_0$ and the positivity of $\Phi_d$, we have
$$
2d-\overline H_0(x)
=\frac{\displaystyle\sum_{y\sim x}\Phi_d^{1/2}(y)}
{\Phi_d^{1/2}(x)}
>0
\qquad\text{for all }x\in\Z^d.
$$
Moreover, $\overline H_0+\mu H_0>0$ under our assumptions. Hence
$$
a(x)
=\bigl(2d-\overline H_0(x)\bigr)
+\bigl(\overline H_0(x)+\mu H_0(x)\bigr)
>2d-\overline H_0(x)>0
\quad\text{for all }x\in\Z^d.
$$
Let $x = \xi$, we obtain
$$ v_{n} (x_1,\xi)  \le \sum_{y\sim \xi} v_{n}(y,\xi) = a(\xi) v_{n}(\xi,\xi) - 1.
$$
Let $x = x_{1}$, we obtain
$$ v_{n} (x_2,\xi)  
\le  \sum_{y\sim x_1} v_{n}(y,\xi) = a(x_1) v_{n}(x_1,\xi).
$$
Proceeding inductively, it follows that
\begin{align*}
v_{n}(z,\xi) = v_{n}(x_l,\xi) 
&\le a(x_{l-1})\cdots a(x_{1})\big(a(\xi)v_{n}(\xi,\xi)-1\big)\\
&\le \frac{\Pi_{i=1}^{\ell-1}a(x_i)(2d-\overline{H}_0(\xi))}{ \overline{H}_0(\xi)+\mu H_0(\xi)} =:C_{z,\xi}.\\
\end{align*}

Finally, we show (\ref{2.1-b-h2}). To this end, we only need to show that  
$v_{n}(x,\xi)< v_{n}(\xi,\xi)$ for $x\not=\xi$ when $\mu\geq 0$. 
If not, there exists $x_0\in \Omega_n \setminus \{\xi\}$ such that 
$v_n(x_0,\xi)=\max_{x\in\Z^d} v_n(x,\xi)\geq v_n(\xi,\xi). $
Let $\cO_+= \big\{x\in\Z^d\setminus\{ \xi\}:\, v_n(x,\xi)=v_n(x_0,\xi)\big\}$.
Since $v_n=0$ in $\Z^d\setminus \Omega_n$,
there exist $x^*\in \cO_+$ and $y\not\in \cO_+$ such that $y\sim x^*$.
Then,
$ -\Delta  v_n(x^*,\xi)>0 $ and $\cL_{\mu}v_n(x^*,\xi)=-\Delta  v_n(x^*,\xi)+\mu H_0(x^*) v_n(x^*,\xi)>0$ when $\mu\geq 0$.
This contradicts $\cL_{\mu}v_n(x^*,\xi) =0$ for $x^* \not= \xi$.
 Then it Then it follows that \eqref{2.1-b-h2} holds.
\hfill$\Box$\bigskip

\noindent{\bf Proof of Theorem \ref{pr fun-fw}. }  \textit{Existence and minimality.} We obtain the fundamental solution $\Phi_{d,\mu}(\cdot,\xi)$ by passing to the limit of finite-domain solutions $v_{n}(\cdot,\xi)$. From Lemma \ref{lm fun-bf-1}, we see that, for each fixed $x$, the mapping $n\mapsto v_{n}(\cdot,\xi)$ is increasing and $\{v_{n}(x)\}_{n}$ has a uniform upper bound $C_{x,\xi}$, which is independent of $n$.
Therefore, we can define the pointwise limit 
 $$\Phi_{d,\mu}(x,\xi):=\lim_{n\to+\infty} v_n(x,\xi)<+\infty,\quad \forall x\in \Z^d. $$
For fixed $x\in \Z^d$, all its neighbors $y \sim x$ are contained in $B_{n}(\bar \xi)$  for sufficiently large $n$. Since 
$$
 \cL_\mu  v_{n}(x) = (2d+\mu H_{0}(x))v_{n}(x)-\sum_{y\sim x} v_{n}(y)=\delta_{\xi} \quad \text{ in } B_{n}(\bar\xi),
$$
and the above expression only involves a finite number of terms in the summation for each fixed $x$, we may pass to the limit $n\to\infty$ and obtain
$$\cL_\mu \Phi_{d,\mu}(x,\xi) =\delta_{\xi},\quad \forall x\in\Z^d.  $$
Note that 
$$v_{n}>0 \,\text{ in }  B_{n}(\bar{\xi}) \quad\text{ and } \quad \Phi_{d,\mu}(\cdot, \xi)=\lim_{n\to\infty}v_{n}(\cdot, \xi)\geq v_n(\cdot,\xi),$$ 
the comparison principle yields
$$\Phi_{d,\mu}(\cdot, \xi) >0 \quad\text{in }  \Z^d.$$
Futhermore, for any positive solution $u$ of \eqref{eq 5.1-fw1}, we have
$ u(x) \geq v_{n}(x,\xi)$ for every $n$.
Passing to the limit gives 
$u(x)\geq \Phi_{d,\mu}(x,\xi),$
which proves that $\Phi_{d,\mu}(\cdot,\xi)$ is the minimal positive solution. This minimality also shows that the limit is independent of the center $\bar\xi$ used for the exhaustion.
        
\smallskip     
\noindent  \textit{Symmetry:} 
For arbitrary fixed $x,y\in \Z^d$, take $n$ sufficiently large such that $x,y\in \Omega_n$ and $v_{n}(\cdot,x), v_{n}(\cdot,y)$ are fundamental solutions of the problem \eqref{bdd fundamental} with point source $\delta_{x}$ and $\delta_{y}$ respectively. The discrete Green formula in Lemma \ref{Green formula} gives 
\begin{align*}
 v_{n}(x,y) 
 &= \sum_{z\in \Omega_{n}} (\cL_{\mu}v_{n}(\cdot,x))(z) v_n(z,y)\\
 &=\mathcal{E}_{\Omega_{n}}(v_{n}(\cdot,x),v_{n}(\cdot,y)) + \mu \sum_{z\in \Omega_{n}}H_0(z)v_{n}(z,x)v_{n}(z,y)\\
 &=\mathcal{E}_{\Omega_{n}}(v_{n}(\cdot,y),v_{n}(\cdot,x)) + \mu \sum_{z\in \Omega_{n}}H_0(z)v_{n}(z,y)v_{n}(z,x)\\
 &= \sum_{z\in \Omega_{n}} (\cL_{\mu}v_{n}(\cdot,y))(z) v_n(z,x)\\
  &=v_{n}(y,x).
\end{align*}
Taking $n\to \infty$, we obtain
$$\displaystyle \Phi_{d,\mu}(x,y) = \lim_{n\to\infty} v_{n}(x,y) = \lim_{n\to\infty}v_{n}(y,x) = \Phi_{d,\mu}(y,x).$$

\smallskip
\noindent \textit{ Bounds:}
{\it The case $(b_1)$:  $H_0 \in C(\Z^d)$ verify  $(\cH_1)$ and $\mu > -1$.} \\
Now we construct a suitable supersolution for $v_n$.
Set
   $$ \bar u_1(x):=\rho(x)^{\tau_-}-\rho(x)^{\tau_{-}-\frac{1}{2}\epsilon_{0}}. $$
Then $\bar{u}_1 \sim \rho(x)^{\tau_-} \text{ as } |x|\to+\infty,$ 
and
 \begin{align*} 
	\cL_{\mu} \bar u_1(x)&=-\beta_{0}(\tau_--\frac12\epsilon_0 )\rho(x)^{\tau_--2 -\frac12\epsilon_0}+ O(\rho(x)^{\tau_--2-\epsilon_0}) , 
\end{align*}
where $-\beta_{0}(\tau_--\frac12\epsilon_0 )>0$.
For some $r_0\geq  1+2|\xi|$ large enough, we derive that 
\begin{equation*}
	\bar u_1(x)>0,\qquad \cL_{\mu}  \bar u_1(x)   >0 \quad {\rm in}\ \Z^d\setminus B_{r_0}(0). 
\end{equation*} 
Choose sufficietly large $n$ such that $B_{r_0}(0) \subsetneq \Omega_{n}$ and define $D_{n}:= \Omega_{n} \setminus B_{r_0}(0).$ Then $\partial D_{n} = \partial \Omega_{n} \cup \delta B_{r_0}(0)$. 
The estimate \eqref{2.1-b-h1-1} implies that 
$$ v_n(x,\xi)\leq C \quad{\rm on}\  \delta B_{r_0}(0),  $$
and this holds for every $n$. Besides, $\bar{u}_1 >0 $ on  $\delta B_{r_0}(0)$. Therefore, there exists $t_0>1$ such that 
$$t_0\bar u_1(x)\geq v_n(x,\xi)\quad  \,\quad{\rm on}\  \delta B_{r_0}(0).$$
Moreover, $v_{n} =0$ on $\partial B_{n}(\bar{\xi})$ while $\bar{u}_1>0$, which implies
$$t_0\bar u_1(x)\geq v_n(x,\xi)\quad  \,\quad{\rm on}\  \partial B_{n}(\bar{\xi}).$$
By the comparison principle, 
$$v_n(x,\xi)\leq t_0\bar u_1(x)  \quad {\rm in\ }D_n. $$
Taking $n\to \infty$, we have
$$\Phi_{d,\mu}(x,\xi)\leq t_0\bar u_1(x)\quad {\rm in}\ \, \Z^d\setminus B_{r_0}(0).   $$
Since $\bar{u}_1 \sim \rho(x)^{\tau_-}$ as $|x| \to\infty$,, we conclude that
\begin{align} \label{boun-1}
\Phi_{d,\mu}(x,\xi)\leq C\rho(x)^{\tau_-(\mu)}\quad {\rm in}\ \,  \Z^d\setminus B_{r_0}(0). 
\end{align}

For subsolution, 
choose
$0<\eta_1<\epsilon_0$, $0<\eta_2<\tau_+-\tau_-,
$
and define
$$u_1:=\rho^{\tau_-}+\rho^{\tau_--\eta_1},
 \qquad
U_\varepsilon:=u_1-\varepsilon\rho^{\tau_-+\eta_2},
 \quad  \varepsilon>0.
$$
By \eqref{est-1-im},
$$
 \cL_\mu u_1
 =\beta_0(\tau_--\eta_1)\rho^{\tau_--\eta_1-2}
 +O(\rho^{\tau_--2-\epsilon_0}),
$$
where
$
 \beta_0(\tau_--\eta_1)
   =-\eta_1(\tau_+-\tau_-+\eta_1)<0.
$
Moreover,
$
 \beta_0(\tau_-+\eta_2)
   =\eta_2(\tau_+-\tau_--\eta_2)>0.
$
Thus, after increasing $r_0$ if necessary, we have
$$
 \cL_\mu u_1\le0,\qquad
 \cL_\mu(\rho^{\tau_-+\eta_2})\ge0
 \qquad\text{in }\Z^d\setminus B_{r_0}(0).
$$
In particular, $\cL_\mu U_\varepsilon\le0$ on this fixed exterior
set for every $\varepsilon>0$.

Since $u_1$ and $\Phi_{d,\mu}(\cdot,\xi)$ are
strictly positive on the finite inner boundary $\delta B_{r_0}(0)$, we may fix
$$
t_1:=\frac12\min_{x\in\delta B_{r_0}(0)}
 \frac{\Phi_{d,\mu}(x,\xi)}{u_1(x)}>0.
$$
Then
$$
 t_1U_\varepsilon\le t_1u_1\le\Phi_{d,\mu}(\cdot,\xi)
 \qquad\text{on }\delta B_{r_0}(0).
$$
Both $r_0$ and $t_1$ are independent of $\varepsilon$. 
On the other hand, for each fixed $\varepsilon>0$, there exists $R_\varepsilon>r_0+2$
such that
$$
 U_\varepsilon
 =\rho^{\tau_-}
     \bigl(1+\rho^{-\eta_1}-\varepsilon\rho^{\eta_2}\bigr)\le0 \qquad \text{for }|x|\ge R_\varepsilon.
$$
Fix $\varepsilon>0$ and let $R>R_\varepsilon+2$. Define
$D_R:=B_R(0)\setminus B_{r_0}(0)$, then
\[
 \mathcal L_\mu\bigl(\Phi_{d,\mu}(\cdot,\xi)-t_1U^\varepsilon\bigr)
 =-t_1\mathcal L_\mu U^\varepsilon\ge0.
\]
The difference is nonnegative on the inner boundary by the choice of
$t_1$, and on the outer boundary because
$\Phi_{d,\mu}>0\ge t_1U^\varepsilon$. Lemma \ref{com lm} therefore yields
\[
 \Phi_{d,\mu}(x,\xi)\ge t_1U_\varepsilon(x)
 \qquad\text{for }x\in D_R.
\]
For fixed $\varepsilon$, letting $R\to\infty$ gives this inequality
throughout $\mathbb Z^d\setminus B_{r_0}(0)$. Taking
$\varepsilon\to 0^+$, we obtain
$$
 \Phi_{d,\mu}(x,\xi)\ge t_1u_1(x)\ge t_1\rho(x)^{\tau_-}
 \qquad\text{for }x\in\mathbb Z^d\setminus B_{r_0}(0).
$$
Together with~\eqref{boun-1}, this proves the required two-sided
estimate outside a finite set when $\mu>-1$. 

\smallskip
\noindent
{\it The case $(b_2)$:  $H_0 \in C(\Z^d)$ verify  $(\cH_1)$ and $\mu = -1$.} 
 Compared with case $(b_1)$, the construction of the supersolution is exactly the same. However, in the critical case $\mu=-1$, the two characteristic exponents coincide and are both
 $\tau_{c}=\frac{2-d}{2}.$
Therefore, the previous construction of the subsolution is no longer applicable. 
This degeneracy requires a modified construction of the subsolution. 
Fix $\gamma\in(0,1).$ 
Set
$$\chi(x) =\rho(x)^{\tau_c} \ell(x)^{\gamma}.$$
Noting that $2\tau_c-2+d = 0$ and $\beta_0(\tau_c) = 0$,
the estimate \eqref{eq:log-expansion} gives,
\begin{align*}
\cL_{-1}\chi(x)
= -4\gamma(\gamma-1)\rho(x)^{\tau_{c}-2}
\ell(x)^{\gamma-2} +O\left(\rho(x)^{\tau_{c}-2-\epsilon_{0}}
\ell(x)^{\gamma}
\right). 
\end{align*} 
Since $\gamma\in(0,1)$, we have
$-4\gamma(\gamma-1)>0$ and $\cL_{-1} \chi >0$ for sufficiently large $|x|$.
For $\varepsilon>0$, define
$$
 U_\varepsilon:=u_1-\varepsilon\chi =\rho^{\tau_c}+\rho^{\tau_c-\eta_1}
-\varepsilon\rho^{\tau_c}\ell^\gamma.
$$
Since $U_{\varepsilon} 
=\rho^{\tau_c}\bigl(1+\rho^{-\eta_1}-\varepsilon\ell^\gamma\bigr)
\le 0$
for sufficiently large $|x|$, $\chi=\rho^{\tau_c}\ell^{\gamma}$ 
provides a suitable correction function required at infinity in the critical case.
The remaining proof of lower bound is similar to the case $(b_1)$.

\smallskip
\noindent
\textit{ The case $(a)$:  $H_0$ verifies $(\cH_0)$ and $\mu > -1$.} 
Fix
$0<\theta<\tau_+-\tau_-$. 
Then
$\beta_0(\tau_-+\theta)>0$
and
$\beta_0(\tau_--\theta)<0$
show that, for a sufficiently large $r_0>2+|\xi|$,
$$
 \cL_\mu(\rho^{\tau_-+\theta})\ge0,
 \qquad
 \cL_\mu(\rho^{\tau_--\theta})\le0
 \qquad\text{in }\Z^d\setminus B_{r_0}(0).
$$
Choose $t_\theta>0$, independent of $n$, so that
$t_\theta\rho^{\tau_-+\theta}\ge v_n(\cdot,\xi)$ on
$\delta B_{r_0}(0)$, using \eqref{2.1-b-h1-1}. Comparison on
$\Omega_n\setminus B_{r_0}(0)$, followed by $n\to\infty$, gives
$$
 \Phi_{d,\mu}(x,\xi)\le t_\theta\rho(x)^{\tau_-+\theta}
 \qquad\text{for }x\in \Z^d\setminus B_{r_0}(0).
$$

For the lower bound, let
$$
 U'_{\varepsilon}
 :=\rho^{\tau_--\theta}-\varepsilon\rho^{\tau_-+\theta},
 \qquad
 t'_\theta:=\frac12\min_{x\in\delta B_{r_0}(0)} \frac{\Phi_{d,\mu}(x,\xi)}{\rho(x)^{\tau_--\theta}}>0.
$$
Then $\cL_\mu U'_{\varepsilon}\le0$ outside $B_{r_0}(0)$,
and
$t'_\theta U'_{\varepsilon}\le\Phi_{d,\mu}(\cdot,\xi)$ on the
inner boundary. For fixed $\varepsilon>0$, 
$
 U'_{\varepsilon}
 =\rho^{\tau_--\theta}(1-\varepsilon\rho^{2\theta})<0
$
for sufficiently large $|x|\ge _{\varepsilon}$. Let $R\ge R_{\varepsilon}+1 $ and compare $U'_{\varepsilon}$ with $\Phi_{d,\mu}$ on finite set $B_{R}(0)\setminus B_{r_0}(0)$. Then, letting $R\to\infty$ and
$\varepsilon\to0^+$ successively yields
$$
 \Phi_{d,\mu}(x,\xi)\ge t'_\theta\rho(x)^{\tau_--\theta}
 \qquad\text{for }x\in \Z^d\setminus B_{r_0}(0).
$$
Both constants and  inner radius may depend on $\theta$ and $\xi$.

\smallskip
\noindent\textit{ Fixed-pole estimates and decay at infinity.}
For each fixed $\xi$,
$$
 (1+|\xi|)^{-1}\rho(x-\xi)\le\rho(x)
\le(1+|\xi|)\rho(x-\xi).
$$
Thus, the estimates in case $(a)$ imply \eqref{2.1-E-bf1}.
In case $(b)$, since $\Phi_{d,\mu}(\cdot,\xi)$ is positive and finite in $B_{r_0}$, we extend the two-sided estimate to $\Z^d$ by adjusting the constants, which proves \eqref{2.1-E-bf}.
Since $\tau_-<0$, estimate \eqref{2.1-E-bf} implies decay at infinity.
Under $(\mathcal H_0)$, choose
$0<\theta<\min\{-\tau_-,\tau_+-\tau_-\}$ in\eqref{2.1-E-bf1} to obtain the decay.

\smallskip
\noindent\textit{ The uniform upper bound when $\mu\ge0$.}
Set $ \bar u_2(x)= \Phi_{d}(x-\xi), $ 
then
\begin{align*} 
	\cL_\mu  \bar{u}_2(x)   &=\delta_{\xi} +\mu H_0  \bar u_2(x)\geq \delta_{\xi}.
\end{align*}
By the comparison principle, we obtain 
	\begin{align} \label{boun-2}
		\Phi_{d,\mu}(x,\xi)\leq  \Phi_d(x-\xi)\quad {\rm in}\ \, \Z^d. 
	\end{align}
For $d\geq 3$, 
$$
\Phi_{d}(x-\xi) \leq C_d (1+|x-\xi|)^{2-d}.
$$
Thus, \eqref{boun-2} leads to \eqref{2.1-F-bf}. 
    
\smallskip
\noindent \textit{ Uniqueness:}  
Let $G(\cdot,\xi)$ be another positive solution of $\cL_\mu G(\cdot,\xi)=\delta_\xi $ in $\Z^d$,
satisfying the same growth condition as $\Phi_{d,\mu}(\cdot,\xi)$. 
Set $ h:=G(\cdot,\xi)-\Phi_{d,\mu}(\cdot,\xi).$
Then
$\cL_\mu h=0$  in $\Z^d.$
We shall compare \(h\) with a positive global supersolution. In the case
$\mu>-1$, choose $\theta<\eta<\tau_+-\tau_-$, where $\theta>0$ is chosen so that both $G$ and $\Phi_{d,\mu}$ are
$O((1+|x|)^{\tau_-+\theta})$. Then
$ \cL_\mu \rho^{\tau_-+\eta}\ge0$   outside a sufficiently large ball.
 In the critical case $\mu=-1$, we instead take 
 $\chi(x)=\rho^{\tau_c}(x)\ell(x)^\gamma$, where $0<\gamma<1$,
for which $\cL_{-1}\chi\ge0 $ outside a sufficiently large ball.
	
In either case, denote the corresponding exterior supersolution by \(\chi\). Let \(F\subset \Z^d\) be a finite set containing the ball outside of which \(\cL_\mu\chi\ge0\), and set
	\[
	q(x):=\sum_{z\in F}\Phi_{d,\mu}(x,z).
	\]
	Then $\cL_\mu q=1_F. $ Since $\chi$ is bounded in a finite set, we may choose \(M>0\) sufficiently large, such that the function
	$
	 p:=\chi+Mq
	$
	satisfies
	\[
	p>0,\qquad L_\mu p\ge0
	\quad\text{in }\mathbb Z^d.
	\]
	Moreover, by the choice of \(\chi\), we have
	\[
	\frac{|h(x)|}{p(x)}\to0
	\qquad\text{as }|x|\to\infty.
	\]
	Then, for  every \(\varepsilon>0\), there exists \(R_\varepsilon>0\) such that
	\[
	|h(x)|\le \varepsilon p(x)
	\qquad\text{for } |x|\ge R_\varepsilon.
	\]
	For \(R>R_\varepsilon\), set
	$
	w_\pm:=\varepsilon p\pm h
	$
	in the finite ball \(B_R(0)\). Since \(\cL_\mu h=0\) and \(\cL_\mu p\ge0\),
	we have
	\[
	\cL_\mu w_\pm=\varepsilon \cL_\mu p\ge0
	\qquad\text{in }B_R(0).
	\]
	Furthermore, by the choice of $R_\varepsilon$,
	$ w\pm\ge0$ on $\partial B_R(0)$.
	The comparison principle
	yields
	$w_\pm\ge0$ in $B_R(0)$.
	Letting $R\to+\infty$, we get
	\[
	-\varepsilon p\le h\le \varepsilon p
	\qquad\text{in }\mathbb Z^d.
	\]
	Letting \(\varepsilon\to0^+\), we obtain \(h\equiv0\). Hence
	\[G(\cdot,\xi)=\Phi_{d,\mu}(\cdot,\xi),\]
	which proves the uniqueness.

	\begin{corollary}\label{cr comp}
		Let the assumptions of  Theorem  \ref{pr fun-fw} hold and  $\Omega$ be a connected {\bf infinite set} of $\Z^d$.   
		Let $u: \Omega\cup\partial\Omega\to \R$ be a function verifying 
		\begin{equation}\label{eq 2.1 cm-ub}
			\left\{\arraycolsep=1pt
			\begin{array}{lll}
			\ \	\cL_\mu  u     \geq  0   \quad
				{\rm on}\ \  \Omega , \\[2mm]
				\phantom{  -\  }
				\ \ u\geq 0 \quad  {\rm   in}\ \,      \partial\Omega, 
				\\[2mm]		\displaystyle\liminf_{\substack{x\in\Omega, \,|x|  \to+\infty}} \frac{u(x)}{\Phi_{d,\mu}(x,0)}\geq0. 
			\end{array}
			\right.
		\end{equation}
		Then $u\geq 0$ in $\Omega$. Furthermore, either $u\equiv 0$ in $\Omega$ or  $u>0$ in $\Omega$. 
		
		In particular, if $$\Phi_{d,\mu}(x,0) \asymp |x|^{\tau_-(\mu)}\quad \quad \text{as } |x|\to\infty,$$then the last condition may be replaced by 
		\[
		\liminf_{\substack{x\in\Omega,\, |x|\to\infty}}
		u(x)|x|^{-\tau_-(\mu)}\ge0.
		\]
		If $\Omega=\Z^d$, we omit the boundary condition.
	\end{corollary}
	\noindent
	{\bf Proof.} 
	Fix \(\varepsilon>0\) and set
	\[
	w_\varepsilon(x):=u(x)+\varepsilon\Phi_{d,\mu}(x,0).
	\]
	Since
	\[
	\cL_\mu \Phi_{d,\mu}(\cdot,0)=\delta_{0}\ge0,
	\]
	we have
	\[
	\cL_\mu w_\varepsilon\ge0
	\qquad \text{in }\Omega.
	\]
	Moreover, the condition at infinity implies that \(w_\varepsilon\ge0\) on
	\(\Omega\setminus B_R(0)\) for \(R\) sufficiently large. Together with
	\(u\ge0\) on \(\partial\Omega\) and \(\Phi_{d,\mu}>0\), this gives
	\(
	w_\varepsilon\ge0
	\text{ on the boundary layer of } \Omega\cap B_R(0).
	\)
	Applying the finite-domain comparison principle Lemma \ref{com lm} to $w_{\epsilon}$ in \(\Omega\cap B_R(0)\), and
	then letting \(R\to\infty\), we obtain
	\[
	u(x)\ge -\varepsilon\Phi_{d,\mu}(x,0)
	\qquad \text{for all }x\in\Omega.
	\]
	Letting \(\varepsilon\to0^+\), we conclude that \(u\ge0\) in \(\Omega\).
	
	Finally, if \(u(x_1)=0\) for some \(x_1\in\Omega\), then
	\[
	0\le \cL_\mu u(x_1)
	=
	-\Delta u(x_1)
	=
	-\sum_{y\sim x_1}u(y)\le0,
	\]
	and hence \(u(y)=0\) for all \(y\sim x_1\). By connectedness, the zero set
	propagates through \(\Omega\). Thus either \(u\equiv0\) in \(\Omega\), or
	\(u>0\) in \(\Omega\). \hfill$\Box$\medskip
	\setcounter{equation}{0}
	\section{Poisson problem in $\Z^d$}
	In this section, we give a classification of the nonnegative solutions of Poisson problem in $\Z^d$. 
	\subsection{Existence and asymptotic profile}
	
	\begin{proposition}\label{pr 3.1}
Suppose  that  
$ \mu \geq -1$,  $H_0$ verifies $(\cH_1)$ and  assume  more that 	 
$H_0<\overline{H}_0$ in $\Z^d$ when $\mu=-1$. 
If $ g\geq 0$ satisfies $\mathcal M_\mu(g)<\infty$, then the Poisson problem
\begin{equation*}
	\cL_\mu v=g\quad\text{in }\Z^d,\qquad v\ge0
\end{equation*}
has a minimal nonnegative solution
\begin{equation*}
		v_g(x)=\sum_{y\in\Z^d}\Phi_{d,\mu}(x,y)g(y).
\end{equation*}
It is positive if $g\not\equiv 0$.

Furthermore, if 
$ g(x)\asymp \rho(x)^{\tau-2}$, $|x|\geq R_{g}$, for some $R_g \ge 0$, then the minimal positive solution $v_g$ has the following asymptotic profile
 	\begin{equation}\label{eq:poisson-profile-branches}
	v_g(x)\asymp
		\begin{cases}
\rho(x)^{\tau}, &\mu>-1,\quad\tau_{-}<\tau<\tau_+,\\
\rho(x)^{\tau_-}\ell(x),
&\mu>-1,\quad\tau=\tau_-,\\
\rho(x)^{\tau_-},&\mu\ge-1,\quad\tau<\tau_-.
       \end{cases}
\end{equation}
\end{proposition}
\smallskip

We first establish the existence of the minimal nonnegative solution, and then provide upper bound estimates for the solution under different conditions. The lower bound estimates are discussed in Lemma \ref{existence poisson}; we state the lower bound estimates separately because they also hold for nonnegative supersolutions of inequality $\cL_{\mu} u \ge f \ge 0$.\medskip
	
\noindent{\bf Proof of Proposition \ref{pr 3.1}. } \textit {Existence and minimality:}  
    The existence can be established via the variational method; here, we use the Green function to give the explicit form of the solution. 
    Set 
    $$v_{n,g}(x):= \sum_{y\in B_{n}} v_{n}(x,y) g(y).$$
    Then 
\begin{equation*}
  \begin{cases}
			\mathcal L_\mu v_{n,g}=g & \text{in }B_n,\\
			v_{n,g}=0 & \text{in }\mathbb Z^d\setminus B_n.
 \end{cases}
\end{equation*}
It follows by the comparison principle that  the mapping $n\mapsto v_{n,g}$ is strictly increasing. Define 
$$v_{g}(x) = \lim_{n\to\infty} v_{n,g}(x) = \sum_{y\in\Z^d}\Phi_{d,\mu}(x,y)g(y).$$
Using the estimate \eqref{2.1-E-bf} and the symmetry of Green function, the assumption $\mathcal{M}_{\mu}(g)<\infty$ implies that, for fixed $x$
 $$ \sum_{y\in\Z^d}\Phi_{d,\mu}(x,y)g(y)  
    = \sum_{y\in\Z^d}\Phi_{d,\mu}(y,x)g(y) 
    \leq \sum_{y\in\Z^d}C_{x} \rho(y) ^{\tau_-}g(y) <\infty.$$
Therefore, the summation is finite and $v_g$ is well defined.
For each fixed $x$, the identity
		\[
		(2d+\mu H_0(x))v_
        {n,g}(x)-\sum_{z\sim x}v_{n,g}(z)=g(x)
		\]
holds for all sufficiently large $n$. Since each vertex has
only finitely many neighbours, passing to the limit gives
	$$
		\mathcal L_\mu v_g=g\,\ \text{ in }\mathbb Z^d.
	$$
    
Furthermore, if the Poisson problem  has a  nonnegative solution $u$, then the finite comparison principle yields $v_{n,g}\leq u$, which implies that  $v_{g}\leq u$. Thus, $v_g$ is a minimal nonnegative solution. If $g\not\equiv 0$, choose $y_0$ such that $g(y_0)>0$. Then $v_g(x) \ge \Phi(x,y_0)g(y_0)>0$ is strictly positive for every $x\in \Z^d$.
	
	\smallskip
    
\noindent  \textit {Upper bound:}  
It can be verified that under each of the branch conditions in \eqref{eq:poisson-profile-branches}, $\mathcal{M}_{\mu}(g) < \infty$, which allows the existence of a minimal positive solution $v_g$. To show the upper bound of $v_g$, we only need the upper bound of $g$ that $g(x)\le C_0\rho(x)^{\tau-2}$, when $|x|\ge R_g$.	\smallskip

\noindent{\textit {Case 1: $\mu>-1, \tau\in(\tau_-,\tau_+)$.} } 
 By estimate \eqref{est-1}, there exist a sufficiently large $R\geq R_g$ and a large $A$ such that
    $$A\cL_\mu \rho(x)^\tau\ge \frac{A}{2}\beta_0(\tau)\rho(x)^{\tau-2}
    \ge g(x) \quad\ \text{ on } \Z^d\setminus B_R.$$
Set $w_1(x)=\sum_{y\in B_{R}(0)}\Phi_{d,\mu}(x,y),$
and choose
$$  
B \geq 1+ \max_{z\in B_R(0)} \big[g(z)-A \cL_{\mu} \rho^{\tau} (z) \big]_+,$$
and define
$\overline{U} :=A\rho^{\tau}+ B w_1 .$
The choice of $B$ gives $$\cL_{\mu}\overline{U} = A\cL_{\mu}\rho^{\tau}\geq C_0 \rho(x)^{\tau-2}\geq g,\qquad |x|> R, $$
and
$$\cL_{\mu}\overline{U} = A\cL_{\mu}\rho^{\tau}+B\geq g,\qquad |x|\leq R. $$
Since $\overline{U}$ is a global positive supersolution and $v_g$ is a minimal positive solution such that 
\[
		v_g(x)\le \overline{U}(x)\le C\rho(x)^{\tau}
		\quad\text{for }x\in \Z^d.
\]
\noindent{\textit {Case 2: $\mu>-1, \tau=\tau_-$.} } 
 By estimate \eqref{eq:log-expansion}, 
\[
\mathcal L_\mu (\rho^{\tau_-}\ell)(x)
=2(d-2)\sqrt{1+\mu}\,\rho^{\tau_- -2}+O(\rho(x)^{\tau_--2-\epsilon_0}\ell(x)),
\]
where $\ell(x)=\log(e+|x|^2).$

Using a similar argument and choose suitable $A_1,T_1$, we obtain a global supersolution $\overline{U}:=A\rho^{\tau_-}\ell+Tw_1$. Therefore,
\[
		v_g(x)\le \overline{U}(x)\le C\rho(x)^{\tau_-}\ell(x)
		\qquad\text{ on }\Z^d.
\]
\noindent{\textit {Case 3: $\mu\ge-1, \tau<\tau_-$.} } 
Choose 
$\sigma\in(\max\{\tau,\tau_--\epsilon_0\},\tau_-)$. 
By estimation \eqref{est-1-im},  there exist a sufficiently large $R\ge R_g$ and $A>0$ such that
$$A\cL_{\mu}(\rho^{\tau_-} - \rho^{\sigma}) \geq -\frac{A}{2}\beta_0(\sigma)\rho(x)^{\sigma -2} \geq g  \quad{\rm for}\ \, |x|\geq R.$$
Set $w_1 = \sum_{y\in B_R(0)}\Phi_{d,\mu}(x,y)$  and 
$B\geq 1+\max_{z\in B_R} [g(z)- \cL_{\mu} (\rho^{\tau_-}-\rho^{\sigma})(z)]_+.$ 
Then,
$$ \cL_{\mu}\overline{U}:= \cL_{\mu} \big[A\big( \rho^{\tau_-} - \rho^{\sigma} \big)+ Bw_1 \big] \ge g \qquad \text{ in } \Z^d.$$
 Therefore, $\overline{U}$ gives an upper bound of order $\tau_-$. Note that this is valid for $\mu=-1$ with $ \tau_- =\tau_+ = \tau_c$.
 \hfill$\Box$\medskip

\begin{lemma}\label{existence poisson}
Suppose  that  
$ \mu \geq -1$,  $H_0$ verifies  $(\cH_1)$, and $u\ge0$ satisfies $\cL_\mu u\ge g$ in $\Z^d$
for some nonzero and nonnegative $g$.
\begin{enumerate}[label=(\roman*)]
\item All the nonnegative supersolutions have a uniform lower bound
	\begin{equation*}
		u(x)\ge c\rho(x)^{\tau_-},\qquad x\in\Z^d.
	\end{equation*}
\item If $\mu>-1$, $\tau \in(\tau_-,\tau_+)$, and
$ g(x)\ge c_0\rho(x)^{\tau-2}$ for $ |x|\ge R_g,$ 
then 
\begin{equation*}
		u(x)\ge c\rho(x)^{\tau},\qquad x\in\Z^d.
	\end{equation*}
\item If $\mu>-1$ and
$ g(x)\ge c_0\rho(x)^{\tau_--2}$ for $ |x|\ge R_g,$
then 
\begin{equation*}
		u(x)\ge c\rho(x)^{\tau_-}\ell(x),\qquad x\in\Z^d.
	\end{equation*}
\end{enumerate}
\end{lemma}

\noindent{\bf Proof. } $(i)$ Notice that  we don't assume $H_0<\overline{H}_0$ in $\Z^d$ when $\mu = -1$. Instead, we introduce the auxiliary potential
\begin{equation}\label{auxiliary potential}
    \widetilde{H}_0 := (1-\frac{1}{2}\rho(x)^{-\epsilon_0}) H_0.
\end{equation}
Then $\widetilde{H}_0$ satisfies $(\cH_1)$ and $0\le \widetilde{H}_0 <\overline{H}_0$ at every vertex, since
$$\overline{H}_0  - \widetilde{H}_0 = (\overline{H}_0  - H_0)+ \frac{1}{2}\rho^{-\epsilon_0} H_0 = O(\overline{H}_0\rho^{-\epsilon_0}).$$
If $\mu > -1$, set $\widetilde{H}_0 = H_0$. In either case, define 
$
    \widetilde{\cL}_{\mu}:= -\Delta -\widetilde{H}_0.
$
Then, $\widetilde{\cL}_{\mu}$ satisfies the assumptions in Theorem \ref{pr fun-fw} and  
$$\widetilde{\cL}_{\mu} u = \cL_{\mu} u + \mu(H_0-\widetilde{H}_0)u \ge g.$$
Let  $g(x_0)>0$ for some $x_0\in\Z^d$.
Applying the
comparison principle Lemma \ref{com lm} and the finite exhaustion to the operatoer $\widetilde{\cL}_{\mu}$ gives
$u \geq  g(x_0) \Phi_{d,\mu}(\cdot,x_0)$ in $\Z^d$.\smallskip

$(ii)$ Choose $\eta\in(\tau_,\,\tau_+)$. By estimate \eqref{est-1}, there exist a sufficiently large $R\ge R_g$ and a small $a>0$ such that  
$$a\cL_{\mu} \rho(x)^{\tau} \le g(x),\qquad \cL_{\mu} \rho(x)^{\eta} \ge 0\qquad \text{ on }\ \Z^d\setminus B_{R}(0).$$
Set $w(x) := \sum_{y\in B_R(0)} \Phi_{d,\mu} (x,y)$, and choose 
$$T \ge 1+\max_{x\in B_R(0)} \big[a\cL_{\mu}\rho(x)^{\tau}+ |\cL_{\mu}\rho(x)^{\eta}|\big]_+.$$
For $0<\varepsilon<1$, define 
$\underline{U}_{\varepsilon} = a\rho(x)^{\tau}- Tw-\varepsilon \rho^\eta.$
The choice of $T$ gives $$\cL_{\mu} \underline{U}_{\varepsilon} \le g \qquad \text{ on } \Z^d.$$ Since $\eta>\tau$, there exist $R_\varepsilon>0$ such that $$\underline{U}_{\varepsilon}(x) < 0 \le u  \qquad \text{ on } \Z^d \setminus B_{R_\varepsilon}(0).$$
For $N>R_{\varepsilon}+2$, one has $\partial B_N(0) \subset \Z^d \setminus B_{R_\varepsilon}$. The finite comparison principle gives 
$ \underline{U_{\varepsilon}}  \le u$  on $B_{N}(0)$.
Therefore, 
$$  \underline{U_{\varepsilon}}  \le u \qquad \text{ on } \Z^d.$$
Letting $\varepsilon \to 0^+$ gives
\begin{align*}
    u &\ge a\rho(x)^{\tau}-Tw(x)\\
    &= \rho(x)^{\tau}(a-T\frac{w(x)}{\rho(x)^{\tau}})  \ge  \frac{a}{2}\rho(x)^{\tau},\quad \forall|x| \ge R_1,
\end{align*}
for some $R_1>0$. Taking
$$c = \min\big\{\frac{a}{2}, \min_{|x|\le R_1} \frac{u(x)}{\rho(x)^{\tau}}\big\},$$
the lower bound of the entire lattice in $(ii)$ follows. \smallskip

$(iii)$ 
At the endpoint $\tau=\tau_-$, replace $\rho^\tau$ by $v=\rho^{\tau_-}\ell$. Since
\[
\mathcal L_\mu v(x)
=2(d-2)\sqrt{1+\mu}\rho^{\tau_- -2}+O(\rho(x)^{\tau_--2-\epsilon_0}\ell(x)),
\]
the same comparison argument, using \(av-Tw-\varepsilon\rho^\eta\) with \(\tau_-<\eta<\tau_+\), yields
\[
u\ge av-Tw_K
\]
after letting \(\varepsilon\to 0^+\). As \(w=O(\rho^{\tau_-})=o(v)\), we obtain \(u\ge c\rho^{\tau_-}\ell\) for sufficiently large \(|x|\), and then extend this bound to the whole lattice after adjusting \(c\).
  \hfill$\Box$\medskip
	
\subsection{Nonexistence}
We now show that the condition $\M_{\mu}(g)<\infty$, already established as sufficient for the solvability of the Poisson problem, is also necessary.
\begin{proposition}\label{pr 3.2} 
Assume that  $d\geq 3$,  $ \mu \geq -1$,  $H_0$ verifies  $(\cH_1)$. 
If $f \geq0$ satisfies
\begin{equation}\label{eq:source-divergence}
	\mathcal{M}_{\mu}(f)=\sum_{x\in\Z^d}f(x)\rho(x)^{\tau_-} =\infty,
\end{equation}
then 
\begin{equation}\label{eq 1.1 EH}
		\cL_\mu u \geq f   \quad
		\text{ in }  \Z^d  , 
        \qquad\qquad
        u\geq 0 \quad 
        \text{ in }   \Z^d
		\end{equation}
has no solution.  
In particular, a sufficient condition for \eqref{eq:source-divergence} is 
\begin{equation}\label{con 2.1-1-asy}
  \liminf_{x\in\Z^d,\,|x|\to+\infty} f(x)|x|_{_Q}^{2 -\tau_+} >0.
\end{equation} 
\end{proposition}
	
\noindent{\bf Proof. } By contradiction, we suppose that  $u_0$ is a nonnegative solution of  (\ref{eq 1.1 EH}). Use the auxiliary potential $\widetilde{H}_0$ defined as \eqref{auxiliary potential} if $\mu = -1$ and $\widetilde{H}_0 = H_0$ if $\mu >-1$.
Then $\widetilde{\mathcal L}_\mu:= -\Delta + \mu\widetilde{H}_0$ satisfies the assumptions in Theorem \ref{pr fun-fw} and
$\widetilde{\mathcal L}_\mu u_0  
\geq f$  in $\Z^d$.
Denote the minimal positive Green kernel of $\widetilde{\mathcal L}_\mu$ by $\widetilde{\Phi}_{d,\mu}$. By symmetry and estimate \eqref{2.1-E-bf},
there exists $c>0$ such that
\[
   \widetilde\Phi_\mu(0,y)
=  \widetilde\Phi_\mu(y,0)
\ge c\,\rho(y)^{\tau_-}
\qquad\text{for every }y\in\mathbb Z^d.
\]
Here the estimate, initially valid outside a finite set, extends to the entire lattice by strict positivity of the Green kernel.
Let $\widetilde{v}_{n}$ be the minimal positive solution of
\begin{equation*}
    \begin{cases}
        \widetilde{\cL}_{\mu} w = \delta_0\quad &\text{ in } B_{n}(0),\\
        w=0\quad 
        &\text{ in } \Z^d\setminus B_{n}(0).
    \end{cases}
\end{equation*}
Define 
$$ \widetilde{v}_{n,f}(x):= 
\sum_{y\in B_{n}(0)} \widetilde{v}_{n}(x,y)f(y).$$
Then
\begin{equation*}
    \begin{cases}
 \widetilde{\cL}_{\mu} \widetilde{v}_{n,f} = f \quad &\text{ in } B_{n}(0),\qquad \\
\widetilde{v}_{n,f}=0\quad 
&\text{ in } \Z^d\setminus B_{n}(0). 
    \end{cases}
\end{equation*}
Since
$$
		\widetilde{\mathcal L}_\mu(u_0-\widetilde{v}_{n,f})\ge0
		\quad\text{in }B_n(0),
		\qquad
		u_0-\widetilde{v}_{n,f}=u_0 \ge0
		\quad\text{on }\partial B_n(0),
$$
the comparison principle yields
		$$
		0\le \widetilde{v}_{n,f}\le u_0
		\qquad\text{in }\mathbb Z^d.
		$$
Then 
\begin{equation*}
	\begin{aligned}
	u(0) \ge \lim_{n\to\infty}
			\sum_{y\in B_{n}(0)}\widetilde v_{n}(0,y)f(y)
	&=\sum_{y\in\Z^d}\widetilde\Phi_{d,\mu}(0,y)f(y)\\
	&\ge c\sum_{y\in\mathbb Z^d}f(y)\rho(y)^{\tau_-}
	= c\,\mathcal M_\mu(f)
    =+\infty,
		\end{aligned} 
\end{equation*}
a contradict, which proves the nonexistence.

From (\ref{con 2.1-1-asy}), there exist $n_1\geq 1$ and $c>0$ such that 
	$$ f(x)\geq c|x|_{_Q}^{\tau_+-2}=c|x|_{_Q}^{-\tau_--d} \qquad\text{ for } |x|\geq n_1,$$
and for $n>n_1$
	\begin{align*}
		\sum_{ n_1\leq |x|_{_Q}\leq n}f(x)|x|_{_Q}^{\tau_-}  \geq  c \sum_{ n_1\leq |x|_{_Q}\leq n} |x|_{_Q}^{-d}
		&= c \sum_{k=n_1}^n  \big|\{x\in \Z^d: |x|_{_Q}=k\}\big|  k^{-d}
		\\[1mm]& \geq  c' \sum_{k=n_1}^n   k^{-1}\to +\infty \quad{\rm as}\ n\to+\infty, 
	\end{align*}
	which  leads to $\mathcal{M}_{\mu}(f)=\infty$.   
    \hfill$\Box$\medskip 
	
\begin{proposition}\label{pr 3.3}
Let $d\ge3$, $\mu<-1$, and
$H_0:\mathbb Z^d\to[0,\infty)$.
The following assertions hold.

\begin{enumerate}
\item[${(i)}$]
Assume that
\begin{equation}\label{prop4.3-assum1}
\liminf_{|x|\to\infty}
\frac{H_0(x)}{\overline H_0(x)}
\ge1.
\end{equation}
For every finite set $K\subset\mathbb Z^d$, set
$\Omega:=\mathbb Z^d\setminus K$.
Then there is no function
$u:\Omega\cup\partial\Omega\to(0,\infty)$ satisfying
\begin{equation}\label{prop4.3-eq1}
  \mathcal L_\mu u\ge0
\quad\text{in }\, \Omega,  
\end{equation}
\item[${(ii)}$]
In particular, if $H_0$ satisfies $(\mathcal H_1)$, then,
for every nonnegative function
$f:\mathbb Z^d\to[0,\infty)$ with $f\not\equiv0$,
the inequality
\begin{equation}\label{prop4.3-eq2}
\mathcal L_\mu u\ge f
\quad\text{in }\, \mathbb Z^d
\end{equation}
admits no nonnegative solution.
\end{enumerate}
\end{proposition}

\noindent{\bf Proof.} We first prove ${(i)}$ by the optimality at infinity of
$\overline H_0$, then assertion ${(ii)}$ follows.
Subsequently, we give an independent proof of ${(ii)}$ based on the nonexistence criterion for the critical Poisson problem established in Proposition \ref{pr 3.1}.\smallskip

\textit{Part (i): }
By contradiction, we assume that there exist a finite set $K_1$ and a positive solution $u$ of \eqref{prop4.3-eq1}. By Lemma \ref{picone-ineq}, 
\begin{equation}\label{picone-4.3}
    \mathcal{E}(\varphi)+\mu\sum_{x}H_0(x)\varphi(x)^2 \geq 0,\qquad\forall\varphi\in C_c(\Z^d\setminus K_1).
\end{equation}
Under assumption \eqref{prop4.3-assum1}, we choose $\delta>0$ small enough such that  
$$-\mu H_0 \geq (1+2\delta)\overline{H}_0\qquad \text{ in } \Z^d\setminus K_2, $$
for some finite subset $K_2\subset\Z^d$.
Set $K= K_1 \cup K_2$.
Optimality at infinity of $\overline{H}_0$ \cite{KPP} implies that, for finite set $K$ and $\delta>0$, there exists a $\varphi_0\in C_c(\Z^d\setminus K)$ satisfying
		\[
		\mathcal{E}(\varphi_0)<(1+\delta)\sum_x\overline{H}_0(x)\varphi_0(x)^2.
		\]
Then
\begin{equation*}
  \begin{aligned}
     &\mathcal{E}(\varphi_0)+\mu\sum_{x}H_0(x)\varphi_0(x)^2\\
     & =\mathcal{E}(\varphi_0)- (1+\delta)\sum_{x}\overline{H}_0(x)\varphi_0(x)^2 
      + \sum_{x} \left(\mu H_0 + (1+\delta)\overline{H}_0\right)\varphi_0(x)^2 \\
     & < 0,
  \end{aligned}
\end{equation*}
which contradicts \eqref{picone-4.3} and proves ${(i)}$.
 \hfill \medskip 
 
\textit{Part  (ii):}
In particular, the assumption $(\mathcal{H}_1)$ satisfies \eqref{prop4.3-assum1}. Suppose that $u\ge 0$ is a solution of \eqref{prop4.3-eq2}. If $u(x_0)=0$, then 
$$0\leq f(x_0) \leq \cL_{\mu}u(x_0) = -\sum_{y\sim x_0} u(y) \le 0.$$
Thus $u$ vanishes at all $y\sim x_0$.
Since $\Z^d$ is connected, $u\equiv 0$ in $\Z^d$, which contradicts $f \not\equiv0$. Therefore, $u$ must be strictly positive in $\Z^d$, which contradicts ${(i)}$ and proves ${(ii)}$.
 \hfill \medskip 

\textit{An alternative proof of (ii).} 
Suppose that there exists a nonnegative solution
$u$ of \eqref{prop4.3-eq2}. Since $f\geq 0$ and $f\not\equiv 0$,
the zero-propagation argument above implies that
$u>0$ on $\mathbb Z^d$.

Set
\[
    \widetilde{\mathcal L}_{-1}:=-\Delta-\widetilde H_0.
\]
where we recall $\widetilde H_0(x)=\left(1-\frac12\rho(x)^{-\epsilon_0}\right)H_0(x)$. 

Then
\[
    0\leq \widetilde H_0(x)<\overline H_0(x)
    \qquad\text{for every }x\in\mathbb Z^d.
\]
Moreover,
\[
\begin{aligned}
    \overline H_0(x)-\widetilde H_0(x)
    &=
    \overline H_0(x)-H_0(x)
    +\frac12\rho(x)^{-\epsilon_0}H_0(x) \\
    &=O\!\left(\overline H_0(x)\rho(x)^{-\epsilon_0}\right)
    \qquad\text{as }|x|\to\infty.
\end{aligned}
\]
Thus $\widetilde H_0$ also satisfies $(\mathcal H_1)$, together
with $\widetilde H_0 < \overline{H}_0$ strictly in $\Z^d$.

Define
\[
    \widetilde{f}(x):=f(x)+(-1-\mu)H_0(x)u(x).
\]
Since $\mu<-1$, we have $\widetilde{f}\geq f\geq 0$ and
$\widetilde{f}\not\equiv 0$. Furthermore, we have that 
\[
\begin{aligned}
    \widetilde{\mathcal L}_{-1}u
    &=
    \mathcal L_\mu u
    +\bigl[(-1-\mu)H_0+(H_0-\widetilde H_0)\bigr]u \\
    &\geq f+(-1-\mu)H_0u
    \\&=\widetilde{f}
    \geq f
    \quad\ \text{in }\mathbb Z^d.
\end{aligned}
\]
Applying Lemma~\ref{existence poisson}\,(i) at the critical
parameter $-1$, with potential $\widetilde H_0$, we obtain
a constant $c_1>0$ such that
\[
    u(x)\geq c_1\rho(x)^{\tau_c},
    \qquad\forall x\in\mathbb Z^d.
\]

By $(\mathcal H_1)$, there exist $c_0>0$ and $R>0$ such that
\[
    H_0(x)\geq c_0\rho(x)^{-2},
    \qquad |x|\geq R.
\]
Consequently,
\[
    \widetilde{f}(x)
    \geq (-1-\mu)H_0(x)u(x)
    \geq c_2\rho(x)^{\tau_c-2},
    \qquad |x|\geq R,
\]
where $c_2:=(-1-\mu)c_0c_1>0$.
Since $2\tau_c-2=-d$, it follows that
$$\mathcal M_{-1}(\widetilde{f})
    =\sum_{x\in\mathbb Z^d}\widetilde{f}(x)\rho(x)^{\tau_c}
    \geq c_2 \sum_{\substack{x\in\mathbb Z^d,\, |x|\geq R}}
    \rho(x)^{2\tau_c-2}
=c_2\sum_{\substack{x\in\mathbb Z^d, \,|x|\geq R}}\rho(x)^{-d}
    =+\infty,$$
where the last divergence follows by summing over all the spheres in $\Z^d$.\,
Then a  contradiction arises from the fact that \eqref{prop4.3-eq2} has no solution  by Proposition \ref{pr 3.2} with $\mu=-1$ and $\widetilde{f}$.  
\hfill$\square$
\medskip

\noindent{\bf Proof of Theorem \ref{teo 1-p}. } 
For part ${(i)}$, Proposition~\ref{pr 3.1} yields the implication
${(a)}$ $\Rightarrow$ ${(c)}$, while
${(c)}$ $\Rightarrow$ ${(b)}$ is immediate.
The implication ${(b)}$ $\Rightarrow$ ${(a)}$ follows by contraposition from Proposition~\ref{pr 3.2}; hence all three parts in ${(i)}$ are equivalent.
Assertions~${(ii)}$ and~${(iii)}$ follow directly from Propositions~\ref{pr 3.2} and~\ref{pr 3.3}, respectively.
\hfill$\Box$\medskip 

\begin{remark}[The exactly critical weight]\label{rm:exact-critical}
The operator $\overline L=-\Delta-\overline{H}_0$ is critical on $\Z^d$, with ground state $h=\sqrt{\Phi_d}$, by \cite{KPP,KPPcrit}.
It has no positive Green kernel. More generally, if $g\ge0$ is nonzero, there is no nonnegative $u$ with $\overline L u\ge g$.
Indeed, such a $u$ would be strictly positive, and Lemma \ref{picone-ineq} would give
\[
	\mathcal{E}(\varphi)-\sum_x\overline{H}_0\varphi^2
		\ge\sum_x\frac{g}{u}\varphi^2,
		\qquad \forall \varphi\in C_c(\Z^d).
\]
The nonzero nonnegative weight $g/u$ contradicts the criticality.
In contrast, under the assumption of $H_0<\overline{H}_0$ strictly, the operator
$\cL_{-1}=\overline L+(\overline{H}_0-H_0)$ is subcritical, as also shown by the
positive Green kernel constructed above.
\end{remark}

\setcounter{equation}{0}
\section{Application to semilinear Hardy equations}

In this section, we study
	\begin{equation}\label{eq 1.1-ext}
	\left\{\arraycolsep=1pt
	\begin{array}{lll}
		\cL_\mu  u\geq W u^p \quad
		&{\rm in}\ \  \Z^d, \\[2mm]
		\phantom{-\   }
		u\gneqq 0  \quad  &{\rm   in} \ \  \Z^d, 
	\end{array}
	\right.
\end{equation}
where $d\ge 3, \mu\ge -1,$ and $p>0$.
We also distinguish its nonnegative supersolutions from solutions of the equation
    \begin{equation}\label{eq 1.1-ext-ex}
	\left\{\arraycolsep=1pt
	\begin{array}{lll}
		\cL_\mu  u =  W u^p \quad
		&{\rm in}\ \  \Z^d, \\[2mm]
		\phantom{-\   }
		u\gneqq 0  \quad  &{\rm   in} \ \  \Z^d, 
	\end{array}
	\right.
\end{equation}
in the sublinear case $0<p<1$.

The distinction between equations and supersolution inequalities is essential. For the  Euclidean Lane-Emden equation 
\begin{equation}\label{eq le-0}
    -\Delta_{\R^d} u = u^p\quad {\rm in}\ \, \R^d,
\end{equation}
where $d\ge3, \ p>1$.
Gidas and Spruck \cite{GS81} 
proved the nonexistence of positive entire classical solutions for $p\in(1, \frac{d+2}{d-2})$ by the Pohozaev identity. 
For the exterior domains, the corresponding
threshold is the Serrin exponent $\frac{d}{d-2}$:
the inequality
$-\Delta_{\mathbb R^d}u\geq u^p$ admits no positive supersolution in an exterior domain when $1<p\leq \frac{d}{d-2}$; see \cite{AS,CPZ}. In particular, \eqref{eq le-0} has no positive solution in $\mathbb R^d\setminus\{0\}$ in this range.
When  $p>\frac{d}{d-2},$ \eqref{eq le-0} in $\R^d\setminus\{0\}$ has a family of fast decaying solutions at infinity and a slow decaying solution by phase plane analysis \cite{MP}. 

The classical Hardy inequality and the improved version \cite{BV1} provide a natural framework for elliptic equations
with inverse-square potentials. For the Euclidean Hardy operator 
$
\cL_{\R^d,\mu}   :=-\Delta_{\R^d} +\mu\frac{(d-2)^2}{4|x|^2},
$
weighted estimates for associated linear Poisson problems were established in \cite{W1}, while a weighted distributional framework for nonhomogeneous Hardy--Leray equations was developed in \cite{CQZ}.
For a bounded domain containing the origin, Chen and Zhou \cite{CZ} studied the isolated singularities and existence of positive solutions of
$$
\cL_{\R^d,\,\mu} u=u^p\quad  {\rm in}\ \, \Omega\setminus\{0\}, \qquad u=0\quad{\rm on}\ \, \partial\Omega,
$$
They obtained the corresponding Serrin critical exponent as
$$p_s(\mu) = 1+\frac{4}{(d-2)(1+\sqrt{1+\mu})}.$$
For
$ 1<p<p_s(\mu)$, this framework yields classification results for isolated singularities and existence results for positive singular solutions with sufficiently small positive singular mass.
For the inequality
\[
\mathcal L_{\mathbb R^d,\mu}u\geq u^p
	\quad\text{in }B_r\setminus\{0\},
\]
Brezis, Dupaigne, and Tesei \cite{BDT} proved that no nontrivial nonnegative distributional supersolution exists
when
\[
	-1\leq\mu<0,
	\qquad
	p\geq
	1+\frac{4}{(d-2)(1-\sqrt{1+\mu})}.
\]
This nonexistence result does not require a boundary condition
on $\partial B_r$.
Related questions concerning singularities, classification,
and asymptotic behavior for Hardy-type equations have been investigated in \cite{C,ChVe1,CY,MV,CHW,D1,F0,FM}.
Equations with absorption terms and Hardy potentials exhibit a different classification, studied in \cite{CF}.
 
\smallskip

Our aim is to find the optimal range of $p$ for the nonexistence of supersolution of \eqref{eq 1.1-ext}. 
For $\theta\in\R$, we separate three assumptions:
\begin{align*}
(\cW_\theta^-) &: \quad W\ge0,\qquad \liminf_{|x|\to\infty}W(x)\rho(x)^{-\theta}>0;\\
(\cW_\theta^+) &: \quad W\ge0,\qquad \limsup_{|x|\to\infty}W(x)\rho(x)^{-\theta}<\infty;\\
(\cW_\theta^{\pm}) &: \quad W>0,\qquad
0<\liminf_{|x|\to\infty}W(x)\rho(x)^{-\theta}
\le\limsup_{|x|\to\infty}W(x)\rho(x)^{-\theta}<\infty.
\end{align*}
The lower bound is used for nonexistence, the upper bound for supersolution constructions. Under $(\cW_\theta^{\pm})$,
global positivity and a finite-set adjustment give
$0<c\rho^\theta\le W\le C\rho^\theta$ on the whole lattice, which permits the ordered sublinear barriers below.

For $p\neq1$, define
$$\alpha_{p}:= \frac{2+\theta}{1-p}.$$
The exponent \(\alpha_p\) arises from the formal power balance for the model weight $W=\rho^\theta$: matching the powers of the linear and nonlinear terms gives $\alpha-2=\theta+p\alpha$, hence $\alpha=\alpha_p$ for $p\ne1$.\\

Our main results on the nonexistence are the following. 
\begin{theorem}\label{teo 2}
	Let  $d\ge3$, $\mu\geq -1$, $p>0$. Assume that $H_0$ and $W$ satisfy $(\cH_1)$ and $(\cW_{\theta}^- )$ respectively.
Problem~\eqref{eq 1.1-ext} has no positive supersolution in any of the following cases:
\begin{enumerate}[label=(\roman*)]
\item $0<p<1$ and $\alpha_p\ge\tau_+$;
\item $p>1$ and $\alpha_p \le \tau_-$;
\item $p=1$ and \begin{equation*}
			w_\infty := \liminf_{ |x|  \to+\infty}\frac{W(x)}{ H_0(x)}>\mu+1. 
		\end{equation*}
\end{enumerate}
\end{theorem}

When comparing the nonexistence results for
\eqref{eq 1.1-ext} with their Euclidean counterparts, one should note that 
the origin is a regular vertex on $\Z^d$, whereas the Euclidean Hardy potential is singular there. 
Therefore, the Hardy structure of inverse-square is encoded at infinity in $\Z^d$, and the nonexistence results of positive solutions of  (\ref{eq 1.1-ext})  is very different from that in $\R^d$. 

The next theorem constructs positive supersolutions  in the complementary parameter region. Under the common assumption on $W$, it therefore establishes the sharpness of the preceding nonexistence ranges for $0<p\neq1$.

\begin{theorem}\label{teo cri ex 1}
	Assume that  $d\geq 3$, $\mu \ge -1$, $H_0, W$ verifies $(\mathcal{H}_1)$ and $(\cW_{\theta}^{+})$ respectively, and $H_0(x) < \overline{H}_0(x)$ for $x\in \Z^d$ when $\mu=-1$. Let  $p>0$, $p\ne 1$. Problem \eqref{eq 1.1-ext} has a positive supersolution if one of the following conditions holds:
\begin{enumerate}[label={(\roman*)}]
 \item $0<p<1$ and $\alpha_p < \tau_+$;
 \item $p>1$ and $\alpha_p > \tau_-$.
\end{enumerate}
If $(\cW_{\theta}^{\pm})$ holds in (i), then Problem \eqref{eq 1.1-ext-ex} has a positive solution.
\end{theorem}

\begin{corollary}\label{cor:sharp}
Assume that  $d\geq 3$, $\mu \ge -1$, $H_0, W$ verifies $(\mathcal{H}_1)$ and $(\cW_{\theta}^{\pm})$ respectively, and $H_0(x) < \overline{H}_0(x)$ for $x\in \Z^d$. For  $p>0$, $p\ne 1$,
the problem \eqref{eq 1.1-ext} admits a sharp  dichotomy,
\begin{center}
\renewcommand{\arraystretch}{1.2}
\begin{tabular}{@{}lll@{}}
\toprule
Range & Nonexistence & Existence\\
\midrule
$0<p<1$ & $\alpha_p\ge\tau_+$ & $\alpha_p<\tau_+$\\
$p>1$   & $\alpha_p\le\tau_-$ & $\alpha_p>\tau_-$\\
\bottomrule
\end{tabular}
\end{center}
In the sublinear range, the same condition is necessary and sufficient for
\eqref{eq 1.1-ext-ex} to have a positive solution.
\end{corollary}

The following proposition records the profiles supplied by the existence construction. Keeping them separate from the main criterion in Theorem \ref{teo cri ex 1} distinguishes the solvability boundary from changes of profile within the existence region.

\begin{proposition}\label{existence-profiles}
Let the assumptions in Theorem \ref{teo cri ex 1} holds. Positive supersolution $u$ of \eqref{eq 1.1-ext} can be constructed with the following profiles:
\begin{enumerate}[label=\textup{(\Alph*)},leftmargin=2.7em]
 \item If either $0<p<1$ and $\alpha_p<\tau_-$, or $p>1$ and $\alpha_p>\tau_-$, 
 then  $u\asymp\rho^{\tau_-}.$
 
 \item If $\mu>-1$, $p>0$, $p\neq 1$ and  $\tau_-<\alpha_p<\tau_+$, then $ u\asymp\rho^{\alpha_p}.$
 
 \item If $\mu>-1$, $0<p<1$ and $\alpha_p=\tau_-$, then
        $ u\asymp\rho^{\tau_-}\ell^{1/(1-p)}.$
        
 \item If $\mu>-1$, $p>1$ and $\alpha_p=\tau_+$, then
     $    u\asymp\rho^{\tau_+}\ell^{-1/(p-1)}.$
\end{enumerate}
Under $(W_\theta^\pm)$ and
$0<p<1$, the supersolutions in \textup{(A)}, \textup{(B)}, and \textup{(C)}
can be chosen to solve \eqref{eq 1.1-ext-ex}, with the same profiles.
\end{proposition}

For $0<p<1$, cases \textup{(A)}, \textup{(C)}, and \textup{(B)} partition
the existence region $\alpha_p<\tau_+$ into
\[
 \alpha_p<\tau_-,\qquad
 \alpha_p=\tau_-<\tau_+,\qquad
 \tau_-<\alpha_p<\tau_+,
\]
respectively. For $p>1$, case \textup{(A)} already covers the whole existence
region; cases \textup{(B)} and \textup{(D)} provide additional profiles within it. When $\mu=-1$, i.e., $\tau_-=\tau_+=\tau_c$, only case\textup{(A)} applies.

\subsection{Equivalent classificaitons}
For 
$d\geq3,\ \mu\geq -1$ and $\theta\in\R$, define
\begin{equation*}
	p^*_{\mu,\theta}=  1+\frac{2+\theta}{-\tau_-(\mu)},\qquad\quad  
	p^\#_{\mu,\theta}=
		1+  \frac{2+\theta}{-\tau_+(\mu)}\quad\text{when } \mu\neq0.
\end{equation*}
No value of $p^\#_{\mu,\theta}$ is assigned at $\mu=0$, because $\tau_+=0$.
Note that
$p^*_{\mu,0}=1+\frac{4}{(d-2)(1+\sqrt{ 1+\mu})},$
which coincides with the Serrin critical exponent of the Hardy operator in $\R^d$. For notation simplicity, we omit the subscripts when the parameters are fixed and no ambiguity arises.

Notice that, for $p>0$, $p\ne1$, 
\begin{equation*}
\begin{aligned} 
\alpha_p-\tau_-&=\frac{\tau_-}{1-p}(p-p^*),\\
 \alpha_p-\tau_+&=\frac{\tau_+}{1-p}(p-p^\#)
       \qquad(\mu\ne0).
\end{aligned}
\end{equation*}
These identities translate the main criterion to critical exponents. In particular, $\alpha_p=\tau_-$ is equivalent to $p=p^{*}$,
and, when $\mu\ne0$, $\alpha_p=\tau_+$ is equivalent to $p=p^{\#}$.
Since $\tau_-<0$, the superlinear condition $\alpha_p>\tau_-$ is equivalent
to $p>p^*$ for every $\mu\ge-1$. 
However, for $0<p<1$, the condition
involving $p^\#$ depends on the sign of $\tau_+$.  
At $\mu=-1$, write
\begin{equation*}
 \tau_-=\tau_+=\tau_c,\qquad
 p^*=p^\#=:p_c=1+\frac{2+\theta}{-\tau_c}.
\end{equation*}
The exact correspondence is recorded in Table \ref{tab:criticalcriteria}.
Every entry is restricted to the $p$-range in its column heading.

\begin{table}[!ht]
\centering
\small
\renewcommand{\arraystretch}{1.25}
\caption{The sharp criteria in critical-exponent coordinates.}
\label{tab:criticalcriteria}
\begin{tabular}{@{}lllll@{}}
\toprule
 & \multicolumn{2}{c}{$0<p<1$} & \multicolumn{2}{c}{$p>1$}\\
\cmidrule(lr){2-3}\cmidrule(l){4-5}
Parameter & Existence & Nonexistence & Existence & Nonexistence\\
\midrule
$\mu=-1$   & $p>p_c$ & $p\le p_c$ & $p>p_c$ & $p\le p_c$\\
$-1<\mu<0$ & $p>p^\#$ & $p\le p^\#$ & $p>p^*$ & $p\le p^*$\\
$\mu=0$    & $\theta<-2$ & $\theta\ge-2$ & $p>p^*$ & $p\le p^*$\\
$\mu>0$    & $p<p^\#$ & $p\ge p^\#$ & $p>p^*$ & $p\le p^*$\\
\bottomrule
\end{tabular}
\end{table}

Consequently, whenever defined, $p^\#$ denotes the sublinear threshold, while $p^*$ denotes the superlinear one. Notably, these critical values need not necessarily lie within the corresponding nonlinear intervals. For instance, if $p^* \le 1$, the superlinear existence condition holds for all $p > 1$. The relative positions of these thresholds with respect to 1 are explicitly illustrated in Table \ref{tab:ordering}.

\begin{table}[!ht]
\centering
\small
\renewcommand{\arraystretch}{1.25}
\caption{Ordering of the critical exponents relative to $1$.}
\label{tab:ordering}
\begin{tabular}{@{}llll@{}}
\toprule
Parameter & $\theta<-2$ & $\theta=-2$ & $\theta>-2$\\
\midrule
$\mu=-1$   & $p_c<1$ & $p_c=1$ & $p_c>1$\\
$-1<\mu<0$ & $ p^\#< p^*<1$ & $ p^\#= p^*=1$ & $1< p^*< p^\#$ \\
$\mu=0$    & $ p^*<1$ & $ p^*=1$ & $ p^*>1$\\
$\mu>0$    & $ p^*<1< p^\#$  & $ p^*= p^\#=1$ & $ p^\#<1< p^*$ \\
\bottomrule
\end{tabular}
\end{table}

Moreover, positivity of the critical numbers is determined by
\begin{equation*}
 p^*>0\ \Longleftrightarrow\ \theta>\tau_--2,\qquad
 p^\#>0\ \Longleftrightarrow\
 \begin{cases}
  \theta>\tau_+-2,&-1\le\mu<0,\\
  \theta<\tau_+-2,&\mu>0.
 \end{cases}
\end{equation*}
All parameter intervals are understood within the prescribed range $p>0$, and  within $0<p<1$ or $p>1$ where specified. Since $p^*$ and $p^\#$ are not assumed to be positive, some of these intervals may be empty.

\subsection{Nonexistence}

\noindent {\bf Proof of Theorem \ref{teo 2}.}  
Suppose that $u_0\gneqq0$ is a supersolution of
\eqref{eq 1.1-ext}. If $u_0(x_0)=0$, then
\[
0\le\cL_\mu u_0(x_0)=-\sum_{y\sim x_0}u_0(y)\le0.
\]
Hence $u_0$ vanishes at every neighbor of $x_0$, and connectivity implies $u_0\equiv0$, a contradiction. Hence $u_0>0$ in $\Z^{d}$. By assumption $(\cW_{\theta}^{-})$, there are $c_W>0$ and an integer $n_0\ge1$ such that
\begin{equation}\label{W^-}
    W(x) \ge c_W \rho(x)^{\theta} \quad \text{ for }\, |x| \ge n_0.
\end{equation}
Since $g:=Wu_0^p$ is nonnegative and nonzero, Proposition \ref{pr 3.2} gives 
\begin{equation}\label{poison summable}
\sum_{x\in\Z^d}W(x)u_0(x)^p\rho(x)^{\tau_-}<\infty.
\end{equation}
We shall also use the lower bound from
Lemma \ref{existence poisson}(i), 
\begin{equation}\label{uniform lower bound}
    u_0 \ge c\rho^{\tau_-} \qquad \text{ in } \Z^d.
\end{equation}

\noindent {\textit{Sublinear case (i): $0<p<1$ and $\alpha_p \ge \tau_+$.}}    Assume first that $\mu>-1$. We construct a positive function $h_+$ such that
\begin{equation}\label{h_+}
    \cL_{\mu} h_+ \ge 0 \quad \text{ on } \Z^d,\qquad h_+(x)\asymp \rho(x)^{\tau_+}. 
\end{equation}
Choose $\eta \in (\tau_-, \tau_+)$ with  $\tau_+ - \eta < \epsilon_0$ and define $\chi(x) = \rho(x)^{\tau_+} + \rho(x) ^{\eta}$. Then 
$$\cL_{\mu} \chi = \beta_0(\eta) \rho^{\eta-2} + O(\rho^{\tau_+ - 2 -\epsilon_0}),$$
where $\beta_0(\eta)>0$. Therefore, $\cL_{\mu} \chi(x) \ge 0 $ outside a finite subset $K\subset \Z^d $. Set $w_K(x) := \sum_{y\in K}\Phi_{d,\mu}(x,y),$ we have 
$$\cL_{\mu}w_K = 1_{K}, \qquad w_K(x) = O(\rho(x)^{\tau_-}).$$
Define $$T= 1+\max_{z\in K}(-\cL_{\mu} \chi(z))_+, \qquad h_+: = \chi + Tw_K,$$ 
then $h_+$ satisfies \eqref{h_+}. 

We now use $h_+$ as a weight to construct a power transformation of $u_0$, in order to obtain a linear Poisson inequality whose source term is independent of $u_0$.
Define $v := u_0^{1-p}h_+^{p}$. For each edge $x\sim y$, the arithmetic-geometry mean inequality yields
$$\frac{v(y)}{v(x)} = \left( \frac{u_0(y)}{u_0(x)} \right)^{1-p}\left( \frac{h_+(y)}{h_+(x)} \right)^{p} \le (1-p)\frac{u_0(y)}{u_0(x)} +p \frac{h_+(y)}{h_+(x)}.$$
Summing over all neighbors of $x$, we obtain
$$\frac{\cL_{\mu}v}{v} \ge (1-p) \frac{\cL_{\mu}u_0}{u_0}+p\frac{\cL_{\mu}h_+}{h_+}.$$
Consequently,
$$\cL_{\mu}v \ge (1-p) v Wu_0^{p-1} = (1-p)Wh_+^{p} \quad \text{ on } \Z^d.$$
Proposition \ref{pr 3.2} gives 
\begin{equation*}
\sum_{x\in\Z^d}W(x)h_+(x)^p\rho(x)^{\tau_-}<\infty.
\end{equation*}
On the other hand, when $\alpha_p\ge \tau_+$, \eqref{W^-}, \eqref{h_+} and $\tau_+ + \tau_- = 2-d$ give
\begin{equation*}
\sum_{x\in\Z^d}W(x)h_+(x)^p\rho(x)^{\tau_-} 
\ge c\sum_{|x|\ge n_0}\rho(x)^{\theta+p\tau_++\tau_-} 
= c \sum_{|x|\ge n_0} \rho(x)^{(1-p)(\alpha_p -\tau_+) -d}   
= +\infty, 
\end{equation*}
which gives a contradiction.

If $\mu =-1$, then $\tau_-=\tau_+ = \tau_c=\frac{2-d}{2}$.
 \eqref{W^-} and \eqref{uniform lower bound} give that, when $\alpha_p \ge \tau_c$,
\begin{equation*}
\sum_{x\in\Z^d}W(x)u_0(x)^p\rho(x)^{\tau_-}\ge c\sum_{|x|\ge n_0} \rho(x)^{\theta + (p+1)\tau_c} =c\sum_{|x|\ge n_0} \rho(x)^{(1-p)(\alpha_p-\tau_c)-d} =+\infty,
\end{equation*} 
which contradicts \eqref{poison summable}.\smallskip
 
\noindent {\textit{Superlinear case (ii): $p>1$ and $\alpha_p \le \tau_-$.}}
The inequality in \eqref{eq 1.1-ext} gives 
$$-\Delta u_0 \ge (Wu_0^{p-1}-\mu H_0)u_0\quad \text{ in }\, \Z^d.$$

When $\alpha_p < \tau_-,$
 \eqref{W^-} and \eqref{uniform lower bound} imply that the following inequality holds outside a finite subset,
$$\frac{W(x)u_0(x)^{p-1}}{\overline{H}_0(x)} \ge c \rho(x)^{\theta+(p-1)\tau_-+2}=c \rho(x)^{(p-1)(\tau_- - \alpha_p)}\to \infty  \qquad \text{ as } |x|\to \infty.$$
Combining this with $\displaystyle\lim_{|x|\to\infty} \frac{H_0(x)}{\overline{H}_0(x)} =1$, we obtain that, for some $\delta>0$ and finite $K$,
\begin{equation}\label{too negative potential}
    Wu_0^{p-1} - \mu H_0 \ge (1+\delta)\overline{H}_0 \qquad \text{ on } \Z^d\setminus K.
\end{equation}
Therefore, 
\begin{equation}\label{optimality}
    -\Delta u_0 \ge (1+\delta)\overline{H}_0 u_0\qquad \text{ on } \Z^d\setminus K,
\end{equation}
which contradicts Proposition \ref{pr 3.3}. 

When $\alpha_p = \tau_-$ and $\mu=-1$, \eqref{W^-} and \eqref{uniform lower bound} imply that the following inequality holds outside a finite subset,
$$Wu_0^{p} \rho^{\tau_c} \ge c \rho^{\theta+(p+1)\tau_c} =c \rho^{2\tau_c -2}=c\rho^{-d}.$$
Thus, 
$\sum_{x\in \Z^d} W(x)u_0(x)^{p}\rho(x)^{\tau_c} =\infty$,  which contradicts Proposition \ref{pr 3.2}.

When $\alpha_p = \tau_-$ and $\mu>-1$,
\eqref{W^-} and \eqref{uniform lower bound} give 
$Wu_0^{p} \ge c\rho^{\theta+p\tau_-}=  c\rho^{\tau_--2}$
outside a finite subset.
It follows from Lemma \ref{existence poisson}(iii) that $u_0 \ge c\rho^{\tau_-}\ell$ in $\Z^d$.
Therefore, the following inequality holds outside a finite subset,
$$\frac{Wu_0^{p-1}(x)}{\overline{H}_0(x)} \ge c\rho(x)^{2+\theta+(p-1)\tau_-}\ell(x)^{p-1} = c\ell(x)^{p-1} \to \infty \quad\text{ as } |x|\to \infty.$$ 
Once again, using \eqref{too negative potential} and \eqref{optimality}, we arrive at a contradiction with  Proposition \ref{pr 3.3}.\smallskip

\noindent {\textit{Linear case (iii): $p=1$ and $w_\infty > \mu+1$.}} 
Choose $\delta>0$ small enough such that $\mu+1+2\delta \le A<w_\infty$. Then, outside a finite subset, we have
$$-\Delta u_0 \ge \left(W-\mu H_0 \right)u_0 \ge (1+\delta)\overline{H}_0u_0,$$
which contradict Proposition \ref{pr 3.3}.
\hfill$\square$
\medskip

\subsection{Existence}
 The strategy to prove existence is the following: construct a positive function with the prescribed behavior at infinity, then apply a Green potential to correct its  $L_\mu$-image on a finite set. A final rescaling yields a supersolution.
In the sublinear range, two suitably ordered constant multiples of this function provide a positive lower and upper barrier for the equation.

\subsubsection{Monotone iteration and rescaling method.}
Once we obtain an ordered pair of positive barriers, the following pointwise iteration produces a solution between them.

\begin{lemma}\label{lem:iteration}
Let $\mu\ge-1$, $0\le H_0\le\overline{H}_0$, $W\ge0$ and $p>0$. Suppose that 
$0<\underline u\le\overline u$ and
\begin{equation}\label{barriers}
 \cL_\mu\underline u\le W\underline u^p,\qquad
 \cL_\mu\overline u\ge W\overline u^p
 \quad\text{on }\Z^d.
\end{equation}
Then \eqref{eq 1.1-ext-ex} has a positive solution $u$ satisfying
$\underline u\le u\le\overline u$ on $\Z^d$.
\end{lemma}
\noindent\textbf{Proof.}  
Write
$$\cL_{\mu} v(x) = a_{\mu}(x)v(x) -\sum_{y\sim x}v(y), \qquad a_{\mu}(x):= 2d+\mu H_0(x). $$
Recall that $\overline{H}_0 = -[\Delta \Phi_d ^{\frac12}]\Phi_d^{-\frac{1}{2}}$,  see  \cite[Theorem 7.1]{KPP}, where $\Phi_d >0$ is the positive Green kernel of $-\Delta$. Thus, for $\mu \ge -1$, we obtain
$$  a_{\mu}(x) \ge 2d- \overline{H}_0 = \frac{\sum_{y\sim x} \Phi_d^{\frac{1}{2}}(y)}{\Phi_d^{\frac{1}{2}}(x)}>0.$$
For $v\ge0$, define
$$(Tv)(x) = \frac{\sum_{y\sim x}v(y)+W(x)v(x)^p}{a_{\mu}(x)}.$$
Since $p>0$, $0\le v_1 \le v_2$ implies $Tv_1 \le Tv_2$. 
Moreover, rewriting the inequalities \eqref{barriers} yields
$\underline{u}\le T\underline{u}$ and $\overline{u}\ge T\overline{u}$. Consequencely, whenever $\underline{u} \le v \le \overline{u}$, 
$$ \underline{u} \le T\underline{u} \le Tv \le T\overline{u} \le \overline{u}.$$

Set $u_0=\underline{u}$ and $u_{n+1} = Tu_n$, then $u_0 = \underline{u}\le T \underline{u} = u_1 \le \overline{u}$. If $u_{n-1} \le u_n \le \overline{u}$, then
$$u_n = Tu_{n-1}\le  Tu_n =u_{n+1} \le T\overline{u} \le \overline{u}.$$
By induction, we obtain
$$0<\underline{u} \le u_n \le u_{n+1} \le \overline{u} \quad \text{ on } \Z^d, \qquad \forall n \ge 0.$$
For each fixed $x$, define
$$u(x):=\lim_{n\to \infty} u_n(x),$$ which is pointwise finite and satisfies
 $0<\underline{u}(x) \le u(x) \le \overline{u}(x)$. By definition of  the sequence $ \{u_{n}\}$,
 $$a_{\mu}(x)u_{n+1}(x) = \sum_{y\sim x} u_n(y) + W(x)u_n(x)^p. $$
Each fixed $x$ has only $2d$ neighbors. Passing to the limit yields
$$a_{\mu}(x) u (x) = \sum_{y\sim x} u(y) + W(x) u(x)^p,$$
which is \eqref{eq 1.1-ext-ex}.
\hfill$\square$\\

\begin{lemma}\label{lem:rescaling}
    Let $d\ge 3$, $\mu\ge -1$ and $p>0, p\neq1$. Assume that $0\le H_0\le \overline{H}_0 $ on $\Z^{d}$. Suppose that $U>0$ satisfies
\begin{equation}\label{eq:nonlinearcomparison}
0<c\rho^{\theta} U^{p} \le \cL_{\mu} U \le C\rho^{\theta} U^{p}, 
\end{equation}
for some constant $0<c\le C<\infty.$ Then the following assertions hold.
\begin{enumerate}
\item[\textup{(i)}]
If $W$ satisfies $(\cW_\theta^+)$,
 then $\lambda U$
is a positive supersolution of \eqref{eq 1.1-ext} for sufficiently large $\lambda>0$ when $0<p<1$,
and for sufficiently small $\lambda>0$ when $p>1$.
\item[\textup{(ii)}]
For $0<p<1$, if $W$ satisfies $(\cW_\theta^\pm)$, then the equation \eqref{eq 1.1-ext-ex}
admits a positive solution $u$ satisfying
\begin{equation}\label{solu-profile}
    aU\le u\le bU
    \quad\text{on }\Z^d
\end{equation}
for some constants $0<a<1<b$.
In particular, $u\asymp U$.
\end{enumerate}
\end{lemma}
\noindent\textbf{Proof.}
(i)
Under assumption $(\cW_{\theta}^{+})$, adjustment on a finite set gives $W \le C_{W}\rho^{\theta}$ on $\Z^d$. Thus, there exists $C_+>1$ such that
$$WU^{p} \le C_W\rho^{\theta}U^p\le C_+ \cL_{\mu}U.$$
For $u=\lambda U$, 
$$Wu^{p} = \lambda^p WU^p \le C_+ \lambda^{p-1}\cL_{\mu}u.$$
Choose $\lambda>0$ sufficiently large when $0<p<1$ and $\lambda>0$ sufficiently  small when $p>1$ such that $C_+ \lambda^{p-1} \le 1$.
Then $\cL_{\mu} u \ge Wu^p$.\\
(ii)
 $(\cW_{\theta}^{\pm})$ implies 
$0<c_W \rho^{\theta}\le W \le C_W\rho^{\theta}$ on $\Z^d$. Then \eqref{eq:nonlinearcomparison} gives 
\begin{equation*}  
 m\le\frac{L_\mu U(x)}{W(x)U(x)^p}\le M,
 \qquad x\in\Z^d,
\end{equation*}
for some constants $0<m<M<\infty$.
Choose 
$$a=(1+M)^{-\frac{1}{1-p}}<1,\qquad b=(1+\frac{1}{m})^{\frac{1}{1-p}} > 1.$$
Then 
$$\frac{\cL_{\mu}(aU)}{W(aU)^{p}} \le a^{1-p}M =\frac{M}{M+1} <1,
\qquad 
\frac{\cL_{\mu}(bU)}{W(bU)^p} \ge b^{1-p}m = m+1 >1.$$
Thus, $aU$ is a positive subsolution, $bU$ is a positive supersolution and $0<aU\le bU$. By Lemma \ref{lem:iteration}, we obtain a positive solution of \eqref{eq 1.1-ext-ex} with 
$ aU\le u\le bU $ on $\Z^d. $
In particular, the solution $u$ has the same profile as $U$.  
\hfill $\square$\\
\begin{remark}
(a) The equality  
$$\frac{\cL_{\mu}(tU)}{W(tU)^p} = t^{1-p}\frac{\cL_{\mu}U}{WU^p}, \quad t>0, $$
shows that, for $p>1$, small $t_1>0$ give supersolutions $t_1U$ and large $t_2>0$ give subsolutions $t_2U$.
Therefore, it does not supply an ordered pair $t_2 U\le t_1U$  for Lemma~\ref{lem:iteration}. \\
(b) At $p=1$, the ratio  is unchanged by rescaling.
Consequencely, we assert equality solutions only for $0<p<1$.
\end{remark}

\medskip

\subsubsection{Construction of prescribed profiles}
Recall the expansion  \eqref{eq:log-expansion}: for fixed $\sigma, s \in \R$, with $|x|\to\infty$,
\begin{align*}
 \cL_\mu [\rho(x)^{\sigma}\ell(x)^s]
 = |x|^{\sigma-2}\bigl[\beta_0(\sigma)\ell^s
-2s(2\sigma+d-2)\ell^{s-1}-4s(s-1)\ell^{s-2}\bigr] 
&+O\bigl(|x|^{\sigma-2-\epsilon_0}\ell^s\bigr).
\end{align*}
 
\begin{proposition}\label{prop:logarithmic}
Assume $d\ge3$, $\mu>-1$, $H_0$ and $W$ satisfies $(\cH_1)$  and $(\cW_\theta^+)$ respectively.
\begin{enumerate}[label=\textup{(\roman*)}]
 \item If $0<p<1$ and $\alpha_p=\tau_-$, then \eqref{eq 1.1-ext} has a positive supersolution satisfying
$$u\asymp\rho^{\tau_-}\ell^{1/(1-p)}\qquad\text{on }\Z^d.$$
Under $(\cW_\theta^\pm)$ it can be chosen to solve \eqref{eq 1.1-ext-ex}.
\item If $p>1$ and $\alpha_p=\tau_+$, then
\eqref{eq 1.1-ext} has a positive supersolution satisfying
$$ u\asymp\rho^{\tau_+}\ell^{-1/(p-1)} \quad\text{on }\Z^d.$$
\end{enumerate}
\end{proposition}

\noindent\textbf{Proof.} 
(i) Set $v=\rho^{\sigma} \ell^s$  with $\sigma=\tau_-$ and $s>0$. 
Since $\beta_0(\tau_-)=0$ and $-2s(2\tau_-+d-2)=2s(d-2)\sqrt{1+\mu}>0$, \eqref{eq:log-expansion} implies that,
$$ 
\cL_{\mu} (\rho^{\tau_-}\ell^{s}) \asymp
\rho^{\tau_- -2}\ell^{s-1} >0$$
outside a finite set $K$. 
Notice that $\alpha_p = \tau_-$ implies $\tau_-(1-p)-(\theta+2)=0.$ Therefore,
\begin{equation}\label{ineq:rescaling}
     \frac{\cL_{\mu} v}{\rho^{\theta}v^{p}} \asymp  \rho^{\tau_-(1-p)-2-\theta} \ell^{(1-p)s-1} = \ell^{(1-p)s-1}  \qquad \text{on } \ \Z^d\setminus K.
\end{equation}
 To obtain the desired inequality \eqref{eq:nonlinearcomparison}, we choose $s=\frac{1}{1-p}>0$.
Define 
\begin{equation}\label{w_K}
w_K (x): = \sum_{y\in K} \Phi_{d,\mu}(x,y), \qquad   U:=v+tw_K,
\end{equation}
where $
t:=1+\max_{x\in K} |\cL_{\mu} v(x)|.$
Then, 
$$\cL_{\mu} U >0, \quad U \asymp v=\rho^{\tau_-}\ell^{\frac{1}{1-p}} \qquad \text{ on } \Z^d.$$
On $\Z^d \setminus K$, using\eqref{ineq:rescaling} with $s= \frac{1}{1-p}$, we obtain
$$\cL_{\mu} U =\cL_{\mu}v\asymp \rho^{\theta}v^p \asymp \rho^{\theta}U^p.$$
Since both $\cL_{\mu}$ and $\rho^{\theta}u^{p}$ are positive and finite on $K$, adjusting the constants yields \eqref{eq:nonlinearcomparison} on the whole lattice. The result then follows from Lemma \ref{lem:rescaling}(ii). \\
(ii) Set $v=\rho^{\tau_+}\ell^{-\frac{1}{p-1}}$. Since $-2s(2\sigma+d-2) = \frac{2(d-2)}{p-1}\sqrt{1+\mu}>0$,  
$$\cL_{\mu}v \asymp \rho^{\tau_+ -2} \ell^{-\frac{p}{p-1}}  >0 \qquad \text{on } \ \Z^d\setminus K$$
for some finite $K$, and 
$$
     \frac{\cL_{\mu} v}{\rho^{\theta}v^{p}} \asymp  \rho^{\tau_+(1-p)-2-\theta} \ell^{(1-p)s-1} = 1 \qquad \text{on } \ \Z^d\setminus K.
$$
Defining $w_K$ and $U$ as \eqref{w_K} in (i), we have 
$$\cL_{\mu} U >0, \quad U \asymp v=\rho^{\tau_+}\ell^{-\frac{1}{p-1}} \quad \text{ on }\,  \Z^d,$$
while
$$\cL_{\mu} U \asymp \rho^{\theta}U^p\qquad \text{ on } \Z^d\setminus K.$$
Similarly, adjusting the constants yields \eqref{eq:nonlinearcomparison} on the whole lattice. Lemma \ref{lem:rescaling}(i) yields the asserted positive supersolution. Under the additional assumption $(\cW_{\theta}^{\pm})$, Lemma \ref{lem:rescaling}(ii) yields a positive solution of \eqref{eq 1.1-ext-ex} with the same profile.
\hfill $\square$\\

\begin{proposition} \label{prop:interior}
Assume $d\ge3$, $\mu>-1$, $H_0$ and $W$ satisfies $(\cH_1)$  and $(\cW_\theta^+)$ respectively. Let $p>0$, $p\ne1$, and
suppose
$\tau_-<\alpha_p<\tau_+.$
Then \eqref{eq 1.1-ext} has a positive supersolution satisfying
$$u\asymp\rho^{\alpha_p}\quad\text{on }\Z^d.
$$
If $0<p<1$ and $(\cW_\theta^\pm)$ holds, such a function can be chosen
to solve \eqref{eq 1.1-ext-ex}.
\end{proposition}

\noindent\textbf{Proof.} 
Set $v=\rho^{\alpha_p}$. Since $\alpha_p \in(\tau_-,\tau_+)$,
$\beta_0(\alpha_p)>0$. By \eqref{est-1}, there exists a
finite ball $K$ satisfies
$$
 \cL_\mu v\asymp\rho^{\alpha_p-2}>0
 \qquad\text{on }\Z^d\setminus K.
$$
Set $w_K(x) = \sum_{y\in K} \Phi_{d,\mu} (x,y)$.
Choose
$$
 t:=1+2\max_{x\in K}|\cL_\mu v(x)|,\qquad U:=v+t w_K.
$$
Then 
$$
 \cL_\mu U=L_\mu v+t1_K>0,\qquad
 U\asymp\rho^{\alpha_p}\qquad
 \text{on }\Z^d.
$$
On $\Z^d \setminus K$,
$$ \cL_{\mu} U = \cL_{\mu}v \asymp \rho^{\alpha_p -2}=\rho^{p\alpha_p +\theta} \asymp \rho^{\theta}U^{p}.
$$
Moreover, $\cL_{\mu}U$ and $\rho^{\theta}U^{p}$ are positive on $K$, adjusting the constants yields \eqref{eq:nonlinearcomparison} on the whole lattice. The results then follow from Lemma \ref{lem:rescaling}.
\hfill $\square$\\

\begin{proposition} \label{prop:lowerroot}
Assume  that $d\ge3$, $\mu\ge -1$,  $H_0$, $W$ verifies $(\cH_1)$  and $(\cW_\theta^+)$ respectively, and $H_0(x)< \overline{H}_0$ for every $x\in \Z^d$ when $\mu=-1$. 
Suppose that either $0<p<1$ with $\alpha_p <\tau_-$, or $p>1$ with $\alpha_p>\tau_-$. Then
\eqref{eq 1.1-ext} possesses a positive supersolution satisfying
$$u\asymp\rho^{\tau_-}\quad\text{on }\Z^d.$$
If $0<p<1$ and $(W_\theta^\pm)$ holds, such a function can be chosen
to solve \eqref{eq 1.1-ext-ex}.
\end{proposition}

\noindent\textbf{Proof.}
Observe that the inequality $2+\theta+p\tau_- <\tau_-$ is implied by both $0<p<1$ with $\alpha_p <\tau_-$ and $p>1$ with $\alpha_p>\tau_-$.
Define
$$\tau_p := 2+\theta+p\tau_-,\qquad v:=\rho^{\tau_p}.$$
The exponent $\tau_p$ is chosen so that $\tau_p-2 = \theta +p\tau_-$, the order of the desired nonlinear source of a function with profile $\rho^{\tau_-}$. Since $\beta_0(\tau_p)<0$,
by \eqref{est-1}, there exists a
finite ball $K$ satisfies
$$
 -\cL_\mu v\asymp \rho^{\tau_p-2}>0
 \qquad\text{on }\Z^d\setminus K.
$$
Set $w_K(x) =\sum_{y\in K} \Phi_{d,\mu} (x,y)$. Take $c_K\ge1$ such that $$c_K ^{-1} \rho^{\tau_-} \le w_K \le c_K \rho^{\tau_-} \qquad \text{ on } \Z^d.$$
Define 
$$t = 2c_K + \max_{x\in K} |\cL_{\mu} v|+1,\qquad U:= tw_K -v.$$
Because $\rho \ge 1 $ and $ \tau_p < \tau_-$, 
$v=\rho^{\tau_p} \le \rho^{\tau_-}\le c_K w_K.$
The choice of $t$ gives 
$$\frac{t}{2} w_K \le (t-c_K)w_K \le U \le tw_K.$$
Thus $U>0$ and $U\asymp \rho^{\tau_-}$ on $\Z^d$.
Moreover, $\cL_{\mu}U = t- \cL_{\mu}v>0$ on $K$  and $\cL_{\mu}U = -\cL_{\mu}v > 0 $ outside $K$. Now we consider the nonlinear source. On one hand,
$$\cL_{\mu} U \asymp \rho^{\tau_p-2} = \rho^{p\tau_-+\theta}\asymp \rho^{\theta}U^p.$$
Positivity and finiteness on $K$ extend the estimate to the whole lattice. This proves \eqref{eq:nonlinearcomparison} and the results follow from Lemma \ref{lem:rescaling}.
\hfill $\square$ \\
\medskip

\noindent\textbf{Proof of Theorem \ref{teo cri ex 1} and Proposition \ref{existence-profiles}.}
For $0<p<1$ and $\alpha_p<\tau_+$,
there are three possibilities:
\begin{enumerate}[label=\textup{(\alph*)}]
 \item If $\alpha_p<\tau_-$, Proposition~\ref{prop:lowerroot} gives a supersolution with profile $\rho^{\tau_-}$.
 \item If $\alpha_p=\tau_-<\tau_+$, then $\mu>-1$, and
Proposition~\ref{prop:logarithmic}\textup{(i)} gives a supersolution
with profile $\rho^{\tau_-}\ell^{1/(1-p)}$.
\item If $\tau_-<\alpha_p<\tau_+$, Proposition~\ref{prop:interior}
gives a supersolution with profile $\rho^{\alpha_p}$.
\end{enumerate}
Each proposition yields an equality solution under assumption $(\cW_\theta^\pm)$.
At $\mu=-1$, $\tau_-=\tau_+ $, so the condition
$\alpha_p<\tau_+$ is reduced to condition \textup{(a)}.

For $p>1$ and $\alpha_p>\tau_-$, this is exactly the
superlinear condition in Proposition~\ref{prop:lowerroot}, which therefore covers the whole superlinear existence region. When
$\tau_-<\alpha_p<\tau_+$, Proposition~\ref{prop:interior} supplies the additional power profile. When $\mu>-1$ and $\alpha_p=\tau_+$,
Proposition~\ref{prop:logarithmic}\textup{(ii)} supplies the additional logarithmic profile.
\hfill $\square$
\bigskip\bigskip

\small
\noindent {\bf  Conflicts of interest:} The authors declare that they have no conflicts of interest regarding this work.

\medskip

 \noindent{\small {\bf Acknowledgements:} 
 	H. Chen is supported by NNSF of China [Nos.  12361043].
	
 	B. Hua is supported by NNSF of China (no.12371056), and by Shanghai Science and Technology Program (No. 22JC1400100).

	}

\end{document}